\documentclass[a4paper, oneside]{article}
\usepackage{amsmath}				%
\usepackage{amssymb}				%
\usepackage{amsthm}				%
\usepackage{appendix}				%
\usepackage{color}					%
\usepackage{hyperref}				%
\usepackage{graphicx}				%
\usepackage{mathrsfs}				%
\usepackage{pdfsync}				%
\usepackage{titlesec}				%

\titleformat{\section}[hang]								%
{\bfseries\large}{\thesection.}{0.5em}{}[]					%
\titlespacing*{\section}{0em}{2em}{1.5em}					%

\titleformat{\subsection}[runin]							%
{\bfseries\normalsize}{\thesubsection.}{0em}{\ }[.]			%

\theoremstyle{plain}									%
\newtheorem{thm}{Theorem}[section]					%
\newtheorem{prop}[thm]{Proposition}						%
\newtheorem{lem}[thm]{Lemma}						%
\newtheorem{cor}[thm]{Corollary}						%

\theoremstyle{definition}								%

\theoremstyle{remark}								%
\newtheorem{rem}[thm]{Remark}						%

\numberwithin{equation}{section}						%

\DeclareMathOperator{\hess}{hess}		%hessian matrix
\newcommand{\Romannum}[1]{\uppercase\expandafter{\romannumeral#1\relax}}
\newcommand{\romannum}[1]{\romannumeral#1\relax}

\def \bC {\mathbb C}		%complex
\def \bE {\mathbb E}		%expectation
\def \bN {\mathbb N_+}		%natural
\def \bP {\mathbb P}		%measure
\def \bR {\mathbb R}		%real
\def \bT {\mathbb T}		%torus
\def \bZ {\mathbb Z}		%integer

\def \bfp {\mathbf p}		%
\def \bfr {\mathbf r}		%
\def \bfx {\mathbf x}		%vector real
\def \bfy {\mathbf y}		%vector real

\def \bst {\boldsymbol\tau}			%

\def \cA {\mathcal A}		%anti-symmetric operator
\def \cF {\mathcal F}		%
\def \cH {\mathcal H}		%
\def \cL {\mathcal L}		%generator
\def \cS {\mathcal S}		%symmetric operator
\def \cY {\mathcal Y}		%
\def \fb {\mathfrak b}		%
\def \fp {\mathfrak p}		%
\def \fr {\mathfrak r}		%
\def \fu {\mathfrak u}		%
\def \fv {\mathfrak v}		%

\def \sA {\mathscr A}		%
\def \sF {\mathscr F}		%sigma algebra
\def \sR {\mathscr R}		%
\def \sS {\mathscr S}		%
\def \sW {\mathscr W}		%
\def \sX {\mathscr X}		%

\def \ve {\vec\eta}		%
\def \vn {\vec\nu}			%
\def \vt {\vec\tau}		%

\begin{document}

\pagestyle{plain}
\pagenumbering{arabic}
\bibliographystyle{plain}

\title{Hyperbolic scaling limit of non-equilibrium fluctuations \\for a weakly anharmonic chain}
\author{Lu \textsc{Xu}}
\date{}
\maketitle

%{\let\thefootnote\relax\footnote{Date: \today\ \& File: \jobname.tex}}

\begin{abstract}
We consider a chain of $n$ coupled oscillators placed on a one-dimensional lattice with periodic boundary conditions. 
The interaction between particles is determined by a weakly anharmonic potential $V_n = r^2/2 + \sigma_nU(r)$, where $U$ has bounded second derivative and $\sigma_n$ vanishes as $n \to \infty$. 
The dynamics is perturbed by noises acting only on the positions, such that the total momentum and length are the only conserved quantities. 
With relative entropy technique, we prove for dynamics out of equilibrium that, if $\sigma_n$ decays sufficiently fast, the fluctuation field of the conserved quantities converges in law to a linear $p$-system in the hyperbolic space-time scaling limit. 
The transition speed is spatially homogeneous due to the vanishing anharmonicity. 
We also present a quantitative bound for the speed of convergence to the corresponding hydrodynamic limit. 
\end{abstract}



\section{Introduction}
\label{sec:introduction}

One of the central topics in statistical physics is to derive macroscopic equations in scaling limits of microscopic dynamics. 
For Hamiltonian lattice field, Euler equations can be formally obtained in the limit, under a generic assumption of local equilibrium. 
However, to prove this for deterministic dynamics is known as a difficult task. 
In particular when nonlinear interaction exists, the appearance of shock waves in the Euler equations complicates further the problem. 
In that case, the convergence to the entropy solution is expected. 

The situation is better understood when the microscopic dynamics is perturbed stochastically. 
Proper noises can provide the dynamics with enough ergodicity, in the sense that the only conserved quantities are those evolving with the macroscopic equations \cite{FFL94}. 
The deduction of partial differential equations from the limit of properly rescaled conserved quantities in these dynamics is called \emph{hydrodynamic limit}. 
For Hamiltonian dynamics with noises conserving volume, momentum and energy, Euler equations are obtained under the hyperbolic space-time scale \cite{OVY93,EO14}. 
They are proved by relative entropy technique and restricted to the smooth regime of Euler equations. 

As hydrodynamic limit can be viewed as the law of large numbers in functional spaces, we can go one step further towards the corresponding central limit theorem. 
More precisely, we can investigate the macroscopic time evolution of the fluctuations of the conserved quantities around its hydrodynamic centre. 
If the dynamics is in its equilibrium, these fluctuations are Gaussian and evolve following linearized equations, known as \emph{equilibrium fluctuation}. 
To prove it requires to approximate the space-time variance of the currents associated to the conserved quantities by their linear functions. 
This step is usually called the \emph{Boltzmann--Gibbs principle} \cite{BR84,KL99}. 
For gradient, reversible systems, a general proof of the Boltzmann--Gibbs principle is established in \cite{Chang94} using entropy method. 
In other cases, such as anharmonic Hamiltonian dynamics, the proof usually relies on model-dependent arguments, such as the spectral gap \cite{OS13,OX20}. 

Our main interest is \emph{non-equilibrium fluctuation}, namely the central limit theorem associated to the corresponding hydrodynamic limit for dynamics out of equilibrium. 
Compared to the equilibrium case, the non-equilibrium fluctuation field exhibits long-range space-time correlations, which turns out to be the main difficulty. 
For some dynamics such as symmetric exclusion process (SSEP) and reaction-diffusion model, duality method can be used to control the correlations and obtain the non-equilibrium version of the Boltzmann--Gibbs principle \cite{DFL86,FPV88,BDP92,Ravish92a}. 
For one-dimensional weakly asymmetric exclusion process (WASEP), a microscopic Cole--Hopf transformation \cite{Gartner88} can be applied, instead of the Boltzmann--Gibbs principle, to linearize the currents \cite{DG91,Ravish92b,BG97}. 
While most works deal with the diffusive space-time scale, the totally asymmetric exclusion process (TASEP) is the only model in which non-equilibrium fluctuation is proved under the hyperbolic scale \cite{Rezakh02}. 
Note that all these works are restricted to models with \emph{stochastic integrability} and \emph{single conservation law}. 
%use duality: \cite{PS83} (voter model)
%SSEP: \cite{GJMN18} (SSEP with slow boundary)
%WASEP: \cite{DPS89} (using v-functions)
%use Cole--Hopf, concern KPZ: \cite{DT16} (non-simple WAEP), \cite{CGRS16,CST18} (ASEP(q,j)), \cite{Labbe17} (bridge). 

In the absence of stochastic integrability, non-equilibrium fluctuations are understood for only few models. 
In \cite{CY92}, an Ornstein--Uhlenbeck process is obtained from non-equilibrium fluctuations for one-dimensional Ginzburg--Landau model using logarithmic Sobolev inequality. 
A general derivation of non-equilibrium fluctuations for conservative systems has been largely open for a long period of time since then. 
Recently in \cite{JM18a,JM18b}, a new approach is developed and applied to spatially inhomogeneous WASEP in dimensions $d < 4$. 
Their main tool is relative entropy technique. 
Briefly speaking, Yau's relative entropy inequality \cite{Yau91} says that the derivative of the relative entropy with respect to a given local Gibbs measure is bounded by a dissipative term and an entropy production term. 
In \cite{JM18a,JM18b}, the authors obtain an estimate allowing them to control the entropy production term by the dissipative term, which they called the key lemma. 
An entropy estimate then follows directly from this lemma. 
Using both the lemma and the entropy estimate as input, Boltzmann--Gibbs principle can be proved by a generalized Feyman--Kac inequality \cite[Lemma 3.5]{JM18a}. 

In the present article we study non-equilibrium fluctuations for a Hamiltonian lattice field under the hyperbolic scale. 
Observe that part of the ideas in \cite{JM18a,JM18b} is robust enough to be applied to our model, cf. Section \ref{sec:boltzmann gibbs}. 
Meanwhile, the proof of the key lemma relies heavily on the particular basis of the local functions on the configuration space of WASEP. 
In Section \ref{sec:lemma} we establish a similar estimate for Hamiltonian dynamics. 
The main tools we used are the Poisson equation and the equivalence of ensembles, see Section \ref{sec:ee} and \ref{sec:poisson} for details. 

The microscopic model we study is a noisy Hamiltonian system on one-dimensional lattice space with \emph{vanishing anharmonicity} and \emph{two conservation laws}. 
Precisely speaking, consider a chain of $n$ coupled oscillators, each of them has mass $1$. 
For $i = 0$, $1$, ..., $n$, denote by $(p_i, q_i) \in \bR^2$ the momentum and position of the particle $i$. 
The periodic boundary condition $(p_0, q_0) = (p_n, q_n)$ is applied to the chain. 

\vspace{3pt}
\begin{figure}[htb]
\centering
\includegraphics[width=0.83\textwidth]{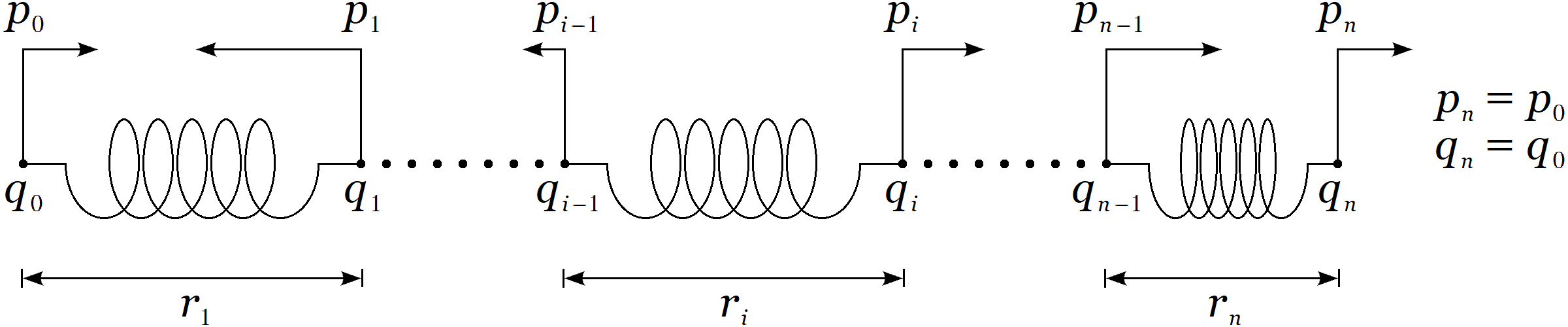}
\caption{Chain of oscillators with periodic boundary}
\end{figure}

\noindent Each pair of consecutive particles $i - 1$ and $i$ is connected by a spring with potential defined by $V(q_i - q_{i-1})$, where $V$ is a nice function on $\bR$. 
With $r_i = q_i - q_{i-1}$ being the relative position, the energy of the chain is given by the Hamiltonian 
$$
  H_n(\bfp, \bfr) = \sum_{i=1}^n \frac{p_i^2}2 + V(r_i). 
$$
When $V$ is quadratic, the corresponding Hamiltonian dynamics is harmonic, and the macroscopic behaviour is known to be purely ballistic. 
We add an anharmonic perturbation to the quadratic potential and define 
$$
  V_\sigma(r) = \frac{r^2}2 + \sigma U(r), \quad \forall r \in \bR, 
$$
where $U$ is a smooth function with good properties, and $\sigma > 0$ is a small parameter which regulates the nonlinearity. 
When $\sigma > 0$ is fixed we say the potential is \emph{anharmonic}, whereas $\sigma \to 0$ is the \emph{weakly anharmonic} case. 

The deterministic Hamiltonian dynamics is perturbed by random, continuous exchange of volume stretch $(r_i, r_{i+1})$ for each $i$, such that $r_i + r_{i+1}$ is conserved. 
The corresponding micro canonical surface is a line, where we add a Wiener process. 
This stochastic perturbation is generated by a symmetric second order differential operator $\cS_{n,\sigma}$ defined later in \eqref{generator}. 
The noise does not conserve $V(r_i) + V(r_{i+1})$, thus breaks the conservation law of energy. 
Notice that the total momentum is naturally another conserved quantity, which is untouched by the noise. 
Similar noise that destroys the energy conservation is also adopted in \cite{Fritz11}. 
Note that the noise in \cite{Fritz11} includes also the exchange of momentum between the nearest neighbour particles. 
In our case the noise on momentum can be dropped, thanks to the linear construction of the momentum fluctuation in the microscopic level. 
We choose the noise in such way that the momentum and volume are the only conserved quantities, hence the equilibrium states are given by canonical Gibbs measures at a fixed temperature $\beta^{-1} > 0$. 

For the anharmonic case, the hydrodynamic equation is
$$
  \partial_t\fp(t, x) = \partial_x\bst_\sigma(\fr(t, x)), \quad \partial_t\fr(t, x) = \partial_x\fp(t, x), 
$$
where $\bst_\sigma$ is the equilibrium tension defined later in \eqref{convex conjugate}. 
It is proved in \cite{OVY93} in smooth regime. 
Denote by $(\fp_\sigma, \fr_\sigma)$ the solution of the equation above. 
Consider the fluctuation field of the conserved quantities along the hydrodynamic equation, given by 
$$
  \frac1{\sqrt n} \sum_{i=1}^n \begin{pmatrix}p_i(t) - \fp_\sigma(t, i/n) \\ r_i(t) - \fr_\sigma(t, i/n)\end{pmatrix}\delta\left(x - \frac in\right). 
$$
Formally, it is expected to converge to a solution of the linearized system 
$$
  \partial_t\tilde\fp_\sigma(t, x) = \bst'_\sigma(\fr_\sigma)\partial_x\tilde\fr_\sigma(t, x), \quad \partial_t\tilde\fr_\sigma(t, x) = \partial_x\tilde\fp_\sigma(t, x). 
$$
Particularly for the equilibrium system, $(\fp_\sigma, \fr_\sigma)$ degenerates to constants and the fluctuation equation is proved in \cite{OX20}, even with the energy conservation and boundary conditions. 
Non-equilibrium fluctuations for anharmonic dynamics remain an open problem. 

We work with the weakly anharmonic case that $\sigma = \sigma_n$ depends on the scaling parameter $n$ in such way that $\sigma_n = o(1)$. 
Similar model with vanishing anharmonicity is also considered in \cite{BGJS18}, where the authors take the FPU-type perturbation $U = r^4$ and the flip-type noise conserving the total energy as well as the sum of the total volume and momentum. 
Although the main interest of \cite{BGJS18} lies in the anomalous diffusion of energy fluctuation, they also prove that under the hyperbolic scale, the time evolution of the fluctuation field of the equilibrium dynamics is governed by a $p$-system. 
Our main result, Theorem \ref{thm:nef}, shows that non-equilibrium fluctuations evolve following a linear $p$-system with spatially homogeneous sound speed, provided that $U$ has bounded second order derivative and $\sigma_n$ decays fast enough. 
This is the first rigorous result obtained for non-equilibrium fluctuations for a Hamiltonian dynamics presenting some level of nonlinearity. 
We also prove a quantitative version of the corresponding hydrodynamic limit in Corollary \ref{cor:quantitative hl}. 

We believe that the macroscopic fluctuation equation proved in this work should be valid with noises acting only on momentum, but the answer is unclear even when the dynamics is in equilibrium. 
Another interesting problem concerns the presentation of boundary conditions in the fluctuation. 
Boundary driven non-equilibrium fluctuations are studied for one-dimensional SSEP in \cite{LMO08,FGN19} and for WASEP in \cite{GLM17}. 
However for Hamiltonian dynamics, it is only studied for equilibrium dynamics \cite{OX20}. 

The article is organized as follows. 
In Section \ref{sec:model and results} we present the precise definition of the microscopic dynamics and state our main results. 
In Section \ref{sec:lemma} we prove the technical lemma, relying on the equivalence of ensembles under inhomogeneous canonical measures and a gradient estimate for the solution of the Poisson equation. 
In Section \ref{sec:relative entropy} we prove the relative entropy estimate Theorem \ref{thm:relative entropy}, based on the technical lemma. 
We also prove the quantitative hydrodynamics limit Corollary \ref{cor:quantitative hl} as an application of Theorem \ref{thm:relative entropy}. 
In Section \ref{sec:boltzmann gibbs} we prove the Boltzmann--Gibbs principle out of equilibrium, along the approach introduced in \cite{JM18a,JM18b}. 
In Section \ref{sec:fdd} and \ref{sec:tightness} we prove the two aspects of the weak convergence of non-equilibrium fluctuations in Theorem \ref{thm:nef}, namely the finite-dimensional convergence and the tightness. 
In Section \ref{sec:ee} and \ref{sec:poisson} we establish the equivalence of ensembles and the gradient estimate for the Poisson equation, respectively. 
Both of them play an important role in the proof of the technical lemma. 
Finally, some auxiliary estimates are collected in the appendix. 
%In Appendix \ref{appendix:tension} we give a brief description of the asymptotic behaviour of the macroscopic tension function for small anharmonicity. 
%In Appendix \ref{appendix:p-system} we restate a classical estimate of the life span of smooth solution of quasi-linear hyperbolic system. 
%It allows us to always deal with smooth solutions of the hydrodynamic limit, as stated in Proposition \ref{prop:smooth regime}. 
%In Appendix \ref{appendix:inequalities} and \ref{appendix:subgaussian}, we collect some useful inequalities necessary for the proof. 
%In Appendix \ref{appendix:yau} we apply Yau's relative entropy inequality to our model and compute the explicit form of the entropy production term. 
%In Appendix \ref{appendix:variation} we restate a variational estimate for the exponential moment of additive functionals of Markov process, introduced in \cite{JM18a,JM18b}. 
%In Appendix \ref{appendix:local clt} we state the local central limit theorem for independent random variables and give a brief proof. 

We close this section with some notations used through the article. 
Let $\bT \sim [0, 1)$ be the one-dimensional torus. 
For a bounded function $f: \bT \to \bR^d$, define 
$$
  |f|_\bT = \sup_\bT |f(x)|_{\bR^d}, \quad \|f\|^2 = \int_\bT |f(x)|_{\bR^d}^2dx. 
$$
Let $\{\varphi_m, m \in \bZ\}$ be the Fourier basis on $\bT$ given by $\varphi_m(x) = e^{2mx\pi i}$. 
For a smooth function $f \in C^\infty(\bT; \bR^2)$ and $k \in \bR$, define 
$$
  \|f\|_k^2 = \sum_{m\in\bZ} \frac{\big|\hat f(m)\big|_{\bC^2}^2}{(1 + m^2)^k}, \quad \hat f(m) = \int_\bT f(x)\overline{\varphi_m(x)}dx. 
$$
Define the Sobolev space $\cH_k(\bT)$ as the closure of $C^\infty(\bT; \bR^2)$ with respect to the norm $\|\cdot\|_k$. 
By a standard dual argument, we can identify $\cH_{-k}(\bT)$ with the space of linear functionals on $C^\infty(\bT; \bR^2)$ which is continuous with respect to $\|\cdot\|_k$. 
For $T > 0$, $C([0, T]; \cH_{-k})$ denotes the set of all continuous trajectories on $[0, T]$ taking values in $\cH_{-k}$, equipped with the uniform topology. 
Also let $C^\alpha([0, T]; \cH_{-k})$ be the subset of $C([0, T]; \cH_{-k})$, consisting of H\"older continuous trajectories with order $\alpha > 0$. 

\section{Microscopic model and main results}
\label{sec:model and results}

For $n \in \bN$, denote by $\bT_n = \bZ/n\bZ$ the one-dimensional discrete $n$-torus, and let $\Omega_n = (\bR^2)^{\bT_n}$ be the configuration space. 
Elements in $\Omega_n$ are denoted by $\ve = \{\eta_i; i \in \bT_n\}$, where $\eta_i = (p_i, r_i) \in \bR^2$. 
Let $U$ be a smooth function on $\bR$ with bounded second order derivative. 
To simplify the arguments, we assume that 
$$
  U(0) = U'(0) = 0, \quad U''(r) \in [-1, 1], \quad \forall r \in \bR. 
$$
For $\sigma \in [0, 1)$, which is supposed to be small eventually, define 
$$
  V_\sigma(r) = \frac{r^2}2 + \sigma U(r), \quad \forall r \in \bR. 
$$
Note that $V_\sigma$ is a smooth function with quadratic growth: 
$$
  \inf V_\sigma'' \ge 1 - \sigma > 0,\quad \sup V''_\sigma \le 1 + \sigma < \infty. 
$$
Define the Hamiltonian $H_{n,\sigma} = \sum_{i\in\bT_n} p_i^2/2 + V_\sigma(r_i)$. 
The corresponding Hamiltonian system is generated by the following Liouville operator 
\begin{align*}
  \cA_{n,\sigma} &= \sum_{i\in\bT_n} (p_i - p_{i-1})\frac\partial{\partial r_i} + \big(V_\sigma'(r_{i+1}) - V_\sigma'(r_i)\big)\frac\partial{\partial p_i} \\
  &= \sum_{i\in\bT_n} (p_i - p_{i-1})\frac\partial{\partial r_i} + (r_{i+1} - r_i)\frac\partial{\partial p_i} + \sigma(U'(r_{i+1}) - U'(r_i)\big)\frac\partial{\partial p_i}
\end{align*}
At each bond $(i, i + 1)$, the deterministic system is contact with a thermal bath at fixed temperature. 
More precisely, fix some $\beta > 0$ and define 
$$
  \cY_i = \frac\partial{\partial r_{i+1}} - \frac\partial{\partial r_i}, \quad \cY_{i,\sigma}^* = \beta\big(V_\sigma'(r_{i+1}) - V_\sigma'(r_i)\big) - \cY_i. 
$$
Notice that $\beta$ is fixed through this article, thus we omit the dependence on it in most cases. 
For $\gamma > 0$, consider the operator $\cL_{n,\sigma,\gamma}$, given by 
\begin{equation}
\label{generator}
  \cL_{n,\sigma,\gamma} = n\big(\cA_{n,\sigma} + \gamma\cS_{n,\sigma}\big), \quad \cS_{n,\sigma} = -\frac12\sum_{i\in\bT_n} \cY_{i,\sigma}^*\cY_i, \quad 
\end{equation}
where $\gamma$ regulates the strength of the noise. 
With an infinite system of independent, standard Brownian motions $\{B^i; i \ge 1\}$, the Markov process generated by $\cL_{n,\sigma,\gamma}$ can be expressed by the solution of the following system of stochastic differential equations: 
$$
  \left\{
  \begin{aligned}
    dp_i(t) =\ &n\big(V'_\sigma(r_{i+1}) - V'_\sigma(r_i)\big)dt, \\%+ \frac{n\gamma}2(p_{i+1} + p_{i-1} - 2p_i)dt, \\
    dr_i(t) =\ &n(p_{i+1} - p_i)dt + \frac{n\beta\gamma}2\big(V'_\sigma(r_{i+1}) + V'_\sigma(r_{i-1}) - 2V'_\sigma(r_i)\big)dt \\
    &+ \sqrt{n\gamma}\big(dB_t^{i-1} - dB_t^i\big), \qquad \forall i \in \bT_n. 
  \end{aligned}
  \right.
$$
It can be treated as the dynamics of the chain of oscillators illustrated in Section \ref{sec:introduction}, rescaled hyperbolically and perturbed with the noise conserving the total momentum $\sum p_i$ as well as the total length $\sum r_i$. 
The total energy $H_{n,\sigma}$ is no longer conserved. 

For $\tau \in \bR$ and $0 \le \sigma < 1$, define the probability measure $\pi_{\tau,\sigma}$ by 
\begin{equation}
\label{one site distribution}
  \pi_{\tau,\sigma}(dr) = \frac1{Z_\sigma(\tau)}e^{-\beta(V_\sigma(r) - \tau r)}dr, 
\end{equation}
where $Z_\sigma(\tau)$ is the normalization constant given by 
$$
  Z_\sigma(\tau) = \int_\bR e^{-\beta(V_\sigma(r) - \tau r)}dr = \int_\bR \exp\left\{-\frac{\beta r^2}2 - \beta\sigma U(r) + \beta\tau r\right\}dr. 
$$
The Gibbs potential $G_\sigma$ and the free energy $F_\sigma$ are then given for each $\tau \in \bR$, $r \in \bR$ by the following Legendre transform 
\begin{equation}
\label{gibbs potential and free energy}
  G_\sigma(\tau) = \frac1\beta\log Z_\sigma(\tau), \quad F_\sigma(r) \triangleq \sup_{\tau\in\bR} \big\{\tau r - G_\sigma(\tau)\big\}. 
\end{equation}
Denote by $\bar r_\sigma$ and $\bst_\sigma$ the corresponding convex conjugate variables 
\begin{equation}
\label{convex conjugate}
  \bar r_\sigma(\tau) = E_{\pi_{\tau,\sigma}} [r] = G'_\sigma(\tau), \quad \bst_\sigma(r) = F'_\sigma(r). 
\end{equation}
Observe that given any finite interval $[r_-, r_+] \in \bR$, 
\begin{equation}
\label{asymptotic tension}
  \big|\bst_\sigma(r) - r\big| \le C\sigma, \quad \big|\bst'_\sigma(r) - 1\big| \le C\sigma, \quad \big|\bst''_\sigma(r)\big| \le C\sigma 
\end{equation}
holds with a uniform constant $C$ for all $r \in [r_-, r_+]$ and sufficiently small $\sigma \ge 0$. 
The details of these asymptotic properties are discussed in Appendix \ref{appendix:tension}. 

For $n \ge 1$, the (grand) Gibbs states of the generator $\cL_{n,\sigma,\gamma}$ are given by the family of product measures $\{\nu_{\bar p,\tau,\sigma}^n; (\bar p, \tau) \in \bR^2\}$ on $\Omega_n$, defined as 
\begin{equation}
\label{gibbs states}
  \nu_{\bar p,\tau,\sigma}^n(d\ve) = \prod_{i\in\bT_n} \sqrt{\frac\beta{2\pi}}\exp\left\{-\frac{\beta(p_i - \bar p)^2}2\right\}dp_i \otimes \pi_{\tau,\sigma}(dr_i). 
\end{equation}
It is easy to see that $\cA_{n,\sigma}$ is anti-symmetric, while $\cS_{n,\sigma}$ is symmetric with respect to the Gibbs states, and for all smooth functions $f$, $g$ on $\Omega_n$, 
$$
  \int_{\Omega_n} f\big(\cS_{n,\sigma}g\big)\;d\nu_{\bar p,\tau,\sigma}^n = -\frac12\int_{\Omega_n} \sum_{i\in\bT_n} \cY_if\cY_ig\;d\nu_{\bar p,\tau,\sigma}^n. 
$$
In particular, $\nu_{\bar p,\tau,\sigma}^n$ is invariant with respect to $\cL_{n,\sigma,\gamma}$. 

\subsection{Weakly anharmonic oscillators}
Pick two positive sequences $\{\sigma_n\}$, $\{\gamma_n\}$ and consider the Markov process in $\Omega_n$ associated to the infinitesimal generator 
\begin{equation}
  \cL_n = \cL_{n,\sigma_n,\gamma_n}, \quad \forall n \ge 1. 
\end{equation}
Basically, we demand that $\sigma_n \to 0$, $\gamma_n \ge 1$ and $\gamma_n = o(n)$. 
These conditions correspond to a weakly anharmonic interaction and assure that the noise would not appear in the hyperbolic scaling limit. 
From here on, we denote 
\begin{equation}
\label{simple notation}
  V_n = V_{\sigma_n}, \quad \cS_n = \cS_{n,\sigma_n}, \quad \bar r_n = \bar r_{\sigma_n}, \quad \bst_n = \bst_{\sigma_n} 
\end{equation}
for short. 
For any fixed $T > 0$, denote by 
$$
  \big\{\ve(t) = (\eta_i(t); i \in \bT_n) \in \Omega_n; t \in [0, T]\big\} 
$$
the Markov process generated by $\cL_n$ and initial distribution $\nu_n$ on $\Omega_n$. 
This is the main subject treated in this article. 
Denote by $\bP_n$, $\bE_n$ the corresponding distribution and expectation on the trajectory space $C([0, T]; \Omega_n)$ of $\ve(\cdot)$, respectively. 

\subsection{Hydrodynamic limit}

We start from the anharmonic case $\sigma_n \equiv \sigma \in (0, 1)$. 
Let $\bP_{n,\sigma}$ denote the law of the Markov process generated by $\cL_{n,\sigma,1}$ and $\nu_n$. 
Assume some profile $\fv \in C^2(\bT; \bR^2)$, such that for any smooth function $g$ on $\bT$, 
\begin{equation}
\label{hydrodynamic initial condition}
  \lim_{n\to\infty} \nu_n \left\{\left|\frac1n\sum_{i\in\bT_n}  g\left(\frac in\right)\eta_i(0) - \int_\bT g(x)\fv(x)dx\right| > \epsilon\right\} = 0, \quad \forall \epsilon > 0. 
\end{equation}
The hydrodynamic limit is then given by the following convergence 
\begin{equation}
\label{hydrodynamic limit}
  \lim_{n\to\infty} \bP_{n,\sigma} \left\{\left|\frac1n\sum_{i\in\bT_n} g\left(\frac in\right)\eta_i(t) - \int_\bT g(x)\begin{pmatrix}\fp_\sigma \\ \fr_\sigma\end{pmatrix}(t, x)dx\right| > \epsilon\right\} = 0, 
\end{equation}
for all $\epsilon > 0$. 
Here $(\fp_\sigma, \fr_\sigma)$ solves the \emph{quasi-linear $p$-system}: 
\begin{equation}
\label{quasi linear p}
  \partial_t\fp_\sigma = \partial_x\bst_\sigma(\fr_\sigma), \quad \partial_t\fr_\sigma = \partial_x\fp_\sigma, \quad (\fp_\sigma, \fr_\sigma)(0, \cdot) = \fv, 
\end{equation}
where $\bst_\sigma = \bst_\sigma(r)$ is the \emph{equilibrium tension} given in \eqref{convex conjugate}. 
Note that the Lagrangian material coordinate is considered as the space variable. 
It is well known that even with smooth initial data, \eqref{quasi linear p} generates shock wave in finite time $T_\sigma$. 
With the arguments in \cite{EO14}, \eqref{hydrodynamic limit} can be proved in its smooth regime, that is, for any $t < T_\sigma$. 

Now we return to the weakly anharmonic case. 
To simplify the notations, denote by $(\fp_n, \fr_n)$ the solution of \eqref{quasi linear p} with $\sigma = \sigma_n$. 
The next proposition allows us to consider only the smooth regime of $(\fp_n, \fr_n)$ for any $T > 0$. 

\begin{prop}
\label{prop:smooth regime}
$\lim_{\sigma\downarrow0} T_\sigma = +\infty$. 
In particular, for any fixed time $T > 0$, we can choose $n_0$ sufficiently large, such that $(\fp_n, \fr_n)$ is smooth on $[0, T]$ for all $n \ge n_0$. 
\end{prop}

Proposition \ref{prop:smooth regime} follows directly from \eqref{asymptotic tension} and Lemma \ref{lem:generation of shock} in Appendix \ref{appendix:p-system}. 
It is not hard to observe that the hydrodynamic equation associated to the weakly anharmonic chain turns out to be the linear $p$-system 
\begin{equation}
\label{linear p}
  \partial_t\fp = \partial_x\fr, \quad \partial_t\fr = \partial_x\fp, \quad (\fp, \fr)(0, \cdot) = \fv. 
\end{equation}
We prove a quantitative convergence in Corollary \ref{cor:quantitative hl} later. 

\subsection{Relative entropy}
For a probability measure $\mu$ on a measurable space $\Omega$, and a density function $f$ with respect to $\mu$, its \emph{relative entropy} is defined by 
\begin{equation}
\label{entropy}
  H(f; \mu) = \int_\Omega f\log fd\mu. 
\end{equation}
Given $T > 0$, let $(\fp_i^n, \fr_i^n)$ be the interpolation of $(\fp_n, \fr_n)$ in \eqref{quasi linear p}: 
$$
  (\fp_i^n, \fr_i^n)(t) = (\fp_n, \fr_n)\left(t, \frac in\right), \quad t \in [0, T],\ i \in \bT_n. 
$$
As discussed before, we assume without loss of generality that $(\fp_n, \fr_n)$ is smooth for $t \in [0, T]$. 
Denote by $\mu_{t,n}$ the \emph{local Gibbs measure} on $\Omega_n$ associated to the smooth profiles $\fp_n(t, \cdot)$ and $\bst_n(\fr_n(t, \cdot))$: 
$$
  \mu_{t,n}(d\ve) = \prod_{i\in\bT_n} \nu_{\fp_i^n,\bst_i^n,\sigma_n}^1(d\eta_i), \quad \bst_i^n = \bst_n(\fr_i^n). 
$$
Let $f_{t,n}$ be the density of the dynamics $\ve(t)$ with respect to $\mu_{t,n}$, and 
$$
  H_n(t) \triangleq H(f_{t,n}; \mu_{t,n}). 
$$
Our first theorem is an estimate on $H_n(t)$, which improves the classical upper bound $H_n(t) \le C(H_n(0) + n)$ for all $t \in [0, T]$. 

\begin{thm}
\label{thm:relative entropy}
There exists a constant $C = C_{\beta,\fv,T}$, such that 
$$
  H_n(t) \le C(H_n(0) + K_n), \quad \forall t \in [0, T],\ n \ge 1, 
$$
where $K_n$ is the deterministic sequence given by 
$$
  K_n = \max \left\{\sigma_n^\frac65\gamma_n^{-\frac15}n^{\frac45}, \gamma_n\right\}. 
$$
\end{thm}

From Theorem \ref{thm:relative entropy}, if $\{\sigma_n\}$, $\{\gamma_n\}$ satisfy that 
\begin{equation}
\label{parameter 2}
  \lim_{n\to\infty} \gamma_n^2n^{-1} = 0, \quad \lim_{n\to\infty} \sigma_n^6\gamma_n^{-1}n^{\frac32} = 0, 
\end{equation}
then $K_n = o(\sqrt n)$ as $n \to \infty$. 
As an application of this observation, we have the following quantitative version of hydrodynamic limit. 

\begin{cor}
\label{cor:quantitative hl}
Assume \eqref{parameter 2} and a constant $C_0$ such that $H_n(0) \le C_0\sqrt n$ for all $n$. 
For any $1 \le p < 2$, $t \in [0, T]$ and smooth function $h: \bT \to \bR^2$, 
$$
  \bE_n \left[\bigg|\frac1n\sum_{i\in\bT_n} h\left(\frac in\right) \cdot \begin{pmatrix}p_i(t) - \fp_i^n(t) \\ r_i(t) - \fr_i^n(t)\end{pmatrix}\bigg|^p\right] \le \frac{C\|h\|^p}{n^{\frac p4}} 
$$
holds with some constant $C = C(\beta, \fv, T, C_0, p)$. 
\end{cor}

Theorem \ref{thm:relative entropy} and Corollary \ref{cor:quantitative hl} are proved in Section \ref{sec:relative entropy}. 

\subsection{Fluctuation field}
By non-equilibrium fluctuation, we mean the fluctuation field of the conserved quantities around its hydrodynamic limit. 
Define the empirical distribution of these fluctuations as 
\begin{equation}
\label{fluctuation field}
  Y_t^n(h) = \frac1{\sqrt n}\sum_{i\in\bT_n} h\left(\frac in\right) \cdot \begin{pmatrix}p_i(t) - \fp_i^n(t) \\ r_i(t) - \fr_i^n(t)\end{pmatrix}, 
\end{equation}
for $t \in [0, T]$, $n \ge 1$ and smooth function $h: \bT \to \bR^2$. 
Notice that the conserved quantities are centred with solutions of \eqref{quasi linear p} instead of \eqref{linear p}. 
Observe that as $n \to \infty$, 
$$
  \|\fp_n(t, \cdot) - \fp(t, \cdot)\| + \|\fr_n(t, \cdot) - \fr(t, \cdot)\| = O(\sigma_n). 
$$
Therefore, $(\fp_n, \fr_n)$ and $(\fp, \fr)$ are indistinguishable in \eqref{fluctuation field} only if $\sqrt n\sigma_n = o(1)$. 
This is not necessarily satisfied in our setting, see \eqref{parameter 2} and \eqref{parameter tem} later. 

By duality, \eqref{fluctuation field} defines a process $\{Y_t^n \in \cH_{-k}(\bT); t \in [0, T]\}$ for $k > 1/2$. 
The major goal of this article is to derive the macroscopic equation of $Y_t^n$. 
Suppose that there is a random variable $Y_0 \in \cH_{-k}$, such that $Y_0^n$ converges weakly to $Y_0$ as $n \to \infty$. 
In the following theorem, we prove that $Y_\cdot^n$ converges weakly to the solution of the a linear $p$-system with homogeneous sound speed under some additional assumptions. 

\begin{thm}
\label{thm:nef}
Assume \eqref{parameter 2} and some $\epsilon > 0$, such that 
\begin{equation}
\label{parameter tem}
  \limsup_{n\to\infty} \sigma_n^2K_n^{3-2\epsilon}n^{\epsilon-1} < \infty, \quad \sup_n H_n(0) < \infty, 
\end{equation}
where $K_n$ is the sequence appeared in Theorem \ref{thm:relative entropy} before. 
For every $T > 0$, $\{(Y_t^n)_{0 \le t \le T}; n \ge 1\}$ converges in law to the unique solution of 
\begin{equation}
\label{homogeneous p}
  \partial_tY(t) = \begin{bmatrix}0 &1 \\1 &0\end{bmatrix}\partial_xY(t), \quad Y(0) = Y_0, 
\end{equation}
with respect to the topology of $C([0, T]; \cH_{-k})$ for $k > 9/2$. 
\end{thm}

\begin{rem}
The additional assumptions in \eqref{parameter tem} are necessary only for the proof of tightness, see Section \ref{sec:tightness}. 
For the convergence of finite-dimensional laws of $Y_t^n$ proved in Section \ref{sec:fdd}, it is sufficient to assume that $H_n(0) = o(\sqrt n)$ and \eqref{parameter 2}. 
\end{rem}

\begin{rem}
In the particular case that $\sigma_n = n^{-a}$, $\gamma_n = n^b$ with $a > 0$, $b \ge 0$, the conditions \eqref{parameter 2} and \eqref{parameter tem} are equivalent to 
$$
  a > \frac15, \quad b \in \big(f_-(a), f_+(a)\big) \cap \left[0, \frac12\right), 
$$
where $f_\pm(a)$ are respectively given by 
$$
  f_-(a) = \frac{7 - 28a}3 \quad\text{and}\quad f_+(a) = \frac{2a + 1}3. 
$$
Hence, if $\sigma_n$ decays strictly faster than $n^{-1/5}$, then the result in Theorem \ref{thm:nef} holds with some properly chosen sequence $\gamma_n$. 
\end{rem}

The proof of Theorem \ref{thm:nef} is divided into two parts. 
In Section \ref{sec:fdd} we show the convergence of finite-dimensional distribution, based on the Boltzmann--Gibbs principle proved in Section \ref{sec:boltzmann gibbs}. 
In Section \ref{sec:tightness} we show the tightness of the laws of $Y_\cdot^n$. 
The weak convergence in Theorem \ref{thm:nef} then follows from the uniqueness of the solution of \eqref{homogeneous p}. 

\section{The main lemma}
\label{sec:lemma}

Fix some $C^1$-smooth function $\tau = \tau(\cdot)$ on $\bT$. 
For each $n \ge 1$, define a product measure $\mu_n$ (dependent on $\tau(\cdot)$, $\sigma_n$) on $\bR^n$ by 
$$
  \mu_n(d\bfr) = \prod_{i\in\bT_n} \pi_{\tau_i^n,\sigma_n}(dr_i), \quad \tau_i^n = \tau\left(\frac in\right). 
$$
Note that $\mu_n$ is the $(r_1, \ldots, r_n)$-marginal distribution of a local Gibbs measure. 
To simplify the notations, let $\langle\;\cdot\;\rangle_{\tau,\sigma}$ denote the integral with respect to $\pi_{\tau,\sigma}$. 
Define 
\begin{equation}
\label{local function}
  \begin{aligned}
    \Phi_i^n(r_i) &= V'_n(r_i) - \langle V'_n \rangle_{\tau_i^n,\sigma_n} - \frac d{dr}\langle V'_n \rangle_{\bst_n(r),\sigma_n}\Big|_{r=r_i^n}(r_i - r_i^n) \\
    &= V'_n(r_i) - \tau_i^n - \bst'_n(r_i^n)\big(r_i - r_i^n\big), 
  \end{aligned}
\end{equation}
where $r_i^n = \bar r_n(\tau_i^n)$ and $\bar r_n$, $\bst_n$ are functions given by \eqref{convex conjugate}, \eqref{simple notation}. 
In this section, we prove an estimate for the space variance associated to $\Phi_i^n$. 

For a probability measure $\mu$ on $\bR^n$ and a density function $f$ with respect to $\mu$, define the Dirichlet form associated to $\cS_{n,\sigma}$ by 
\begin{equation}
\label{dirichlet r}
  D(f; \mu) = \frac12\sum_{i\in\bT_n} \int_{\bR^n} (\cY_if)^2d\mu. 
\end{equation}
For $g \in C^1(\bT)$, define the random local functional 
\begin{equation}
\label{spatial variation}
  W_n(g) = \sum_{i\in\bT_n} g_i^n\Phi_i^n, \quad g_i^n = g\left(\frac in\right). 
\end{equation}

\begin{lem}
\label{lem:main lemma}
For any $\delta > 0$ and density function $f$ with respect to $\mu_n$, there exists a random functional $W_{n,\delta}(g)$ (depending on $\tau(\cdot)$ and $f$), such that 
\begin{equation}
\label{main lemma 1}
 \int f\big[W_n(g) - W_{n,\delta}(g)\big]d\mu_n \le \delta n\gamma_n D\left(\sqrt f; \mu_n\right), 
\end{equation}
where $W_n(g)$ is defined through \eqref{local function} and \eqref{spatial variation} above, and 
\begin{equation}
\label{main lemma 2}
  \int f\big|W_{n,\delta}(g)\big|d\mu_n \le C(1 + M_g)\left[H(f; \mu_n) + \left(1 + \frac1\delta\right)\kappa_n\right], 
\end{equation}
where $C$ is a constant dependent on $\beta$, $|\tau|_\bT$ and $|\tau'|_\bT$, and 
$$
  M_g = |g|_\bT^2 + |g'|_\bT, \quad \kappa_n = \max \left\{\sigma_n^\frac65\gamma_n^{-\frac15}n^{\frac45}, \sigma_n\sqrt n\right\}. 
$$
In particular, if the second limit in \eqref{parameter 2} is satisfied, then $\kappa_n = o(\sqrt n)$. 
\end{lem}

To prove Lemma \ref{lem:main lemma}, we make use of the sub-Gaussian property of the local function $\Phi_i^n$. 
A real-valued random variable $X$ is sub-Gaussian of order $C > 0$, if 
\begin{equation}
  \log E \big[e^{sX}\big] \le \frac{Cs^2}2, \quad \forall s \in \bR. 
\end{equation}
Recall that $V_n = r^2/2 + \sigma_nU$ with $U''$ bounded. 
We have the following lemma. 

\begin{lem}
\label{lem:subgaussian}
For all $n \ge 1$, $i \in \bT_n$, $U'(r_i) - \langle U' \rangle_{\tau_i^n,\sigma_n}$ and $r_i - r_i^n$ are sub-Gaussian of a uniform order dependent only on $\beta$ and $|\tau|_\bT$. 
\end{lem}

The proof of the sub-Gaussian property is direct and is postponed to the end of this section. 
Some general properties of sub-Gaussian variables used hereafter are summarized in Appendix \ref{appendix:subgaussian}. 
Now we state the proof of Lemma \ref{lem:main lemma}. 

\begin{proof}[Proof of Lemma \ref{lem:main lemma}]
Pick some $\ell = \ell(n) \ll n$ which grows with $n$. 
Let 
$$
  g_{i,\ell}^n = g_i^n - \frac1\ell\sum_{j=0}^{\ell-1}g_{i-j}^n, \quad \Phi_{i,\ell}^n = E_{\mu_n} \left[\frac1\ell\sum_{j=0}^{\ell-1} \Phi_{i+j}^n~\bigg|~\sum_{j=0}^{\ell-1} r_{i+j}\right]. 
$$
For each $i \in \bT_n$, denote by $\cY_{i,n}^*$ the adjoint of $\cY_i$ with respect to the inhomogeneous measure $\mu_n$. 
It is easy to see that for smooth $F$, 
\begin{equation}
\label{adjoint}
  \cY_{i,n}^*F = \beta\big(V'_n(r_{i+1}) - V'_n(r_i) - \tau_{i+1}^n + \tau_i^n\big)F - \cY_iF. 
\end{equation}
Let $\psi_{i,\ell}^n = \psi_{i,\ell}^n(r_i, \ldots, r_{i+\ell-1})$ solve the Poisson equation 
\begin{equation}
\label{poisson}
  \sum_{j=0}^{\ell-2} \cY_{i+j,n}^*\cY_{i+j}\psi_{i,\ell}^n = \Psi_{i,\ell}^n, \quad \Psi_{i,\ell}^n = \frac1\ell\sum_{j=0}^{\ell-1} \Phi_{i+j}^n - \Phi_{i,\ell}^n. 
\end{equation}
By Proposition \ref{prop:lipschitz estimate}, $\psi_{i,\ell}^n \in C_b^1(\bR^\ell)$. 
Define the auxiliary functionals 
\begin{align*}
  &W_{n,\ell}^{(1)}(g) = \sum_{i\in\bT_n} g_{i,\ell}^n\Phi_i^n, \quad W_{n,\ell}^{(2)}(g) = \sum_{i\in\bT_n} g_i^n\Phi_{i,\ell}^n, \\
  &W_{n,\ell}^{(3)}(g) = \frac{2(\ell - 1)}{n\gamma_n}\sum_{i\in\bT_n} (g_i^n)^2\sum_{j=0}^{\ell-2} \big(\cY_{i+j}\psi_{i,\ell}^n\big)^2, 
\end{align*}
for each $n \ge 1$, $\ell$ and $i \in \bT_n$. 

Our first step is to observe that for any $\delta > 0$, 
\begin{align*}
  &\int f\left[W_n(g) - W_{n,\ell}^{(1)}(g) - W_{n,\ell}^{(2)}(g) - \frac1\delta W_{n,\ell}^{(3)}(g)\right]d\mu_n \\
  =&\ \int \sum_{i\in\bT_n} g_i^n\sum_{j=0}^{\ell-2} \big(\cY_{i+j}\psi_{i,\ell}^n\big)\big(\cY_{i+j}f\big)d\mu_n - \frac1\delta \int fW_{n,\ell}^{(3)}(g)d\mu_n \\
%  \le&\ \frac{\delta n\gamma_n}{8(\ell - 1)}\int f^{-1}\sum_{i\in\bT_n} \sum_{j=0}^{\ell-2} (\cY_{i+j}f)^2d\mu_n = \delta n\gamma_nD_{(r)}\left(\sqrt f; \mu_n\right). 
  \le&\ \frac{\delta n\gamma_n}{8(\ell - 1)}\int f^{-1}\sum_{i\in\bT_n} \sum_{j=0}^{\ell-2} (\cY_{i+j}f)^2d\mu_n = \delta n\gamma_nD\left(\sqrt f; \mu_n\right). 
\end{align*}
Hence, the strategy is to bound the integrals of the auxiliary functionals by relative entropy together with terms of $\ell$ and $n$, and then optimize the order of $\ell$. 

For the first functional $W_{n,\ell}^{(1)}$, note that $\Phi_i^n = \sigma_n\phi_i^n$, where 
\begin{align*}
  \phi_i^n(r_i) &= U'(r_i) - \langle U' \rangle_{\tau_i^n,\sigma_n} - \frac d{dr}\langle U' \rangle_{\bst_n(r_i^n),\sigma_n}\big(r_i - r_i^n\big), \\
  &= U'(r_i) - \frac{\tau_i^n - r_i^n}{\sigma_n} - \frac{\bst'_n(r_i^n) - 1}{\sigma_n}\big(r_i - r_i^n\big). 
\end{align*}
In view of \eqref{asymptotic tension}, there is a constant $C_{\beta,|\tau|_\bT}$, such that 
$$
  \left|\frac {\bst'_n(r_i^n) - 1}{\sigma_n}\right| \le C_{\beta,|\tau|_\bT}, \quad \forall n \ge 1,\ i \in \bT_n. 
$$
As $\{\phi_i^n; i \in \bT_n\}$ is an independent family, by the entropy inequality \eqref{entropy inequality 3}, 
$$
  \int f\big|W_{n,\ell}^{(1)}\big|d\mu_n \le \frac1\alpha\left(H(f; \mu_n) + \sum_{i\in\bT_n} \log\int e^{\alpha\sigma_n|g_{i,\ell}^n\phi_i^n|}d\mu_n\right), 
$$
for any $\alpha > 0$. 
By Lemma \ref{lem:subgaussian} and direct computation, $\phi_i^n$ is sub-Gaussian of a uniform order $c = c_{\beta,|\tau|_\bT}$. 
Choosing $\alpha_{n,\ell} = (2\sigma_n|g_{i,\ell}^n|)^{-1}$ and applying Lemma \ref{lem:subgaussian absolute}, 
$$
  \int f\big|W_{n,\ell}^{(1)}\big|d\mu_n \le \frac1{\alpha_{n,\ell}}\left[H(f; \mu_n) + \sum_{i\in\bT_n} \left(\log3 + \frac c4\right)\right]. 
$$
As $|g_{i,\ell}^n| \le C|g'|_\bT\ell n^{-1}$ with some universal constant $C$, therefore, 
\begin{equation}
\label{main lemma estimate 1}
  \begin{aligned}
    \int f\big|W_{n,\ell}^{(1)}\big|d\mu_n &\le \frac{C|g'|_\bT\sigma_n\ell}n\big(H(f; \mu_n) + C_1n\big) \\
    &\le C|g'|_\bT\big(H(f; \mu_n) + C_2\sigma_n\ell\big). 
  \end{aligned}
\end{equation}

The second functional $W_{n,\ell}^{(2)}$ is the variance of a canonical ensemble. 
Indeed, $\Phi_{i,\ell}^n = \sigma_n\phi_{i,\ell}^n$, where $\phi_{i,\ell}^n$ is the conditional expectation on the box $(r_i, \ldots, r_{i+\ell-1})$: 
$$
  \phi_{i,\ell}^n = E_{\mu_n} \left[\frac1\ell\sum_{j=0}^{\ell-1} \phi_{i+j}^n~\bigg|~\sum_{j=0}^{\ell-1} r_{i+j}\right]. 
$$
The definition of $\phi_i^n$ suggests that this term can be estimated by the theory of \emph{equivalence of ensembles} presented in Section \ref{sec:ee}. 
First notice that $\{\phi_{i,\ell}^n, i \in \bT_n\}$ is an $\ell$-independent class. 
With \eqref{entropy inequality 3} we obtain that for any $\alpha > 0$, 
$$
  \int f\big|W_{n,\ell}^{(2)}\big|d\mu_n \le \frac1\alpha\left(H(f; \mu_n) + \frac1\ell\sum_{i\in\bT_n} \log\int e^{\alpha\ell\sigma_n|g_i^n\phi_{i,\ell}^n|}d\mu_n\right). 
$$
Since $\phi_{i,\ell}^n$ is sub-Gaussian of order $c$, in view of Lemma \ref{lem:subgaussian absolute}, 
$$
  \int e^{s|\phi_i^n|}d\mu_n \le \frac{1+s}{1-s}e^{\frac{cs}2} \le e, \quad \forall |s| \le A = A(c). 
$$
Hence, Proposition \ref{prop:ee exponential} yields that if $\ell \le O(n^{2/3})$, 
$$
  \int e^{s|\ell\phi_{i,\ell}^n|}d\mu_n \le C_1, \quad \forall |s| \le A' = A'(c), 
$$
with some universal constant $C_1$. 
Choosing $\alpha_n = A'(|g|_\bT\sigma_n)^{-1}$, 
\begin{equation}
\label{main lemma estimate 2}
  \int f\big|W_{n,\ell}^{(2)}\big|d\mu_n \le \frac1{\alpha_n}\left(H(f; \mu_n) + \frac{C_1n}\ell\right) \le C_2|g|_\bT\left(H(f; \mu_n) + \frac{C_1\sigma_nn}\ell\right). 
\end{equation}

For the third functional $W_{n,\ell}^{(3)}$, recall the Poisson equation \eqref{adjoint}--\eqref{poisson}. 
Using the $C^1$ estimate of the Poisson equation in Proposition \ref{prop:lipschitz estimate}, 
$$
  \sum_{j=0}^{\ell-2} \big(\cY_{i+j}\psi_{i,\ell}^n\big)^2 \le C_\beta\ell^4\sup_{\bR^\ell} \left\{\sum_{j=0}^{\ell-2} \big(\cY_{i+j}\Psi_{i,\ell}^n\big)^2\right\}, 
$$
where the supremum is taken over all $(r_i, r_{i+1}, \ldots, r_{i+\ell-1}) \in \bR^\ell$. 
From the definition of $\Psi_{i,\ell}^n$, with $b_i^n = \bst'_n(r_i^n) = \bst'_n(\bar r_n(\tau_i^n))$, 
\begin{align*}
  \cY_{i+j}\Psi_{i,\ell}^n &= \frac1\ell \big(V''_n(r_{i+j-1}) - V''_n(r_{i+j}) - b_{i+j-1}^n + b_{i+j}^n\big) \\
  &= \frac1\ell \big(\sigma_nU''(r_{i+j-1}) - \sigma_nU''(r_{i+j}) - b_{i+j-1}^n + b_{i+j}^n\big). 
\end{align*}
In view of the condition $|U''(r)| \le 1$ and \eqref{asymptotic tension}, 
$$
  \big|\cY_{i+j}\Psi_{i,\ell}^n\big| \le \frac{C\sigma_n}\ell\left(1 + \frac1n\right), 
$$
with some constant $C$ dependent on $|\tau'|_\bT$. 
Therefore, 
\begin{equation}
\label{main lemma estimate 3}
  \int fW_{n,\ell}^{(3)}d\mu_n \le \frac{2CC_\beta(\ell - 1)\sigma_n^2}{n\gamma_n} \sum_{i\in\bT_n} (g_i^n)^2\ell^3 \le \frac{C_1|g|_\bT^2\sigma_n^2\ell^4}{\gamma_n}. 
\end{equation}

Combining \eqref{main lemma estimate 1}--\eqref{main lemma estimate 3}, we obtain that if $\ell \le O(n^{2/3})$, 
\begin{align*}
  &\int f\left|W_{n,\delta}^{(1)}(g) + W_{n,\delta}^{(2)}(g) + \frac1\delta W_{n,\delta}^{(3)}(g)\right|d\mu_n \\
  \le&\ C\big(1 + |g|_\bT^2 + |g'|_\bT\big)\left(H(f; \mu_n) + \sigma_n\ell + \frac{\sigma_nn}\ell + \frac{\sigma_n^2\ell^4}{\delta\gamma_n}\right), 
\end{align*}
holds with some constant $C = C(\beta, |\tau|_\bT, |\tau'|_\bT)$. 
The optimal choice of $\ell$ is 
$$
  \ell(n) = \min\left\{(\sigma_n^{-1}\gamma_nn)^{\frac15},\ n^{\frac12}\right\}. 
$$
Indeed, if $\sigma_n^{-1}\gamma_n > n^{3/2}$, we take $\ell = \sqrt n$, and 
$$
  \sigma_n\ell + \frac{\sigma_nn}\ell + \frac{\sigma_n^2\ell^4}{\delta\gamma_n} = \left(2 + \frac{n^{\frac32}}\delta\frac{\sigma_n}{\gamma_n}\right)\sigma_n\sqrt n < \left(2 + \frac1\delta\right)\sigma_n\sqrt n. 
$$
On the other hand, if $\sigma_n^{-1}\gamma_n \le n^{3/2}$, we take $\ell = (\sigma_n^{-1}\gamma_nn)^{1/5}$, and 
\begin{align*}
  \sigma_n\ell + \frac{\sigma_nn}\ell + \frac{\sigma_n^2\ell^4}{\delta\gamma_n} &= \sigma_n^{\frac65}\gamma_n^{-\frac15}n^{\frac45}\left(\sigma_n^{-\frac25}\gamma_n^{\frac25}n^{-\frac35} + 1 + \frac1\delta\right) \\
  &\le \left(2 + \frac1\delta\right)\sigma_n^{\frac65}\gamma_n^{-\frac15}n^{\frac45}. 
\end{align*}
In consequence, \eqref{main lemma 1}, \eqref{main lemma 2} are in force by defining 
$$
  W_{n,\delta}(g) = W_{n,\ell}^{(1)}(g) + W_{n,\ell}^{(2)}(g) + \frac1\delta W_{n,\ell}^{(3)}(g), 
$$
with $\ell = \ell(n)$ chosen above. 
\end{proof}

Before proceeding to the proof of Lemma \ref{lem:subgaussian}, we discuss the anharmonic case briefly. 
If $\sigma_n \equiv \sigma$ and $\gamma_n = o(n)$, similar argument yields the estimate with $\kappa_n$ replaced by $\kappa'_n = n^{3/5+}$. 
Apparently, it is insufficient for deriving the macroscopic fluctuation, which demands at least $\kappa'_n = o(\sqrt n)$. 
By computing explicitly under Gaussian canonical measure, the upper bounds presented for the first and second auxiliary functionals in the proof of Lemma \ref{lem:main lemma} turn out to be sharp. 
Meanwhile, \eqref{main lemma estimate 3} should be improvable. 
Indeed, by using \eqref{entropy inequality 3}, the left-hand side of \eqref{main lemma estimate 3} is bounded from above with 
$$
  2H(f; \mu_n) + \frac2\ell\sum_{i\in\bT_n} \log\int \exp\bigg\{\frac{\ell^2(g_i^n)^2}{n\gamma_n}\sum_{j=0}^{\ell-2}\big(\cY_{i+j}\psi_{i,\ell}^n\big)^2\bigg\}d\mu_n, 
$$
Therefore, we guess that a nice upper bound of the exponential moment term above could help us take the advantage of the entropy and improve \eqref{main lemma estimate 3}. 

Lemma \ref{lem:subgaussian} is a special case of the next result. 

\begin{lem}
\label{lem:general subgaussian}
Let $V \in C(\bR)$ satisfy $c_-r^2 \le 2V(r) \le c_+r^2$ with two positive constants $c_\pm$. 
For $\tau \in \bR$, let $\pi_\tau$ be a probability measure on $\bR$ given by 
$$
  \pi_\tau = e^{-V(r) + \tau r - G(r)}dr, \quad G(r) = \log\int_\bR e^{-V(r) + \tau(r)}dr. 
$$
If $F$ is a measurable function on $\bR$ such that $|F(r)| \le c|r|$ with constant $c$, then $F - E_{\pi_\tau} [F]$ is sub-Gaussian of order $C = C(\tau, c, c_\pm)$ under $\pi_\tau$. 
Furthermore, $C$ is uniformly bounded for all the coefficients in any compact intervals. 
\end{lem}

\begin{proof}
Notice that for all $\tau \in \bR$, 
$$
  e^{G(\tau)} \ge \int_\bR \exp\left\{-\frac{c_+r^2}2 + \tau r\right\} = \frac{\sqrt{2\pi}}{2c_+}\exp\left\{\frac{\tau^2}{2c_+}\right\}. 
$$
For any $t$ such that $0 < t < c_-/(2c^2)$, 
\begin{align*}
  E_{\pi_\tau} \big[\exp(tF^2)\big] &\le e^{-G(\tau)}\int_\bR \exp\left\{-\frac{(c_- - 2tc^2)r^2}2 + \tau r\right\}dr \\
  &\le \frac{c_+}{c_- - 2tc^2}\exp\left\{\frac{\tau^2}2\left(\frac1{c_- - 2tc^2} - \frac1{c_+}\right)\right\}. 
\end{align*}
Denote $F_* = F - E_{\pi_\tau} [F]$. 
By convexity, for all $t \ge 0$, 
$$
  E_{\pi_\tau} \big[\exp(tF_*^2)\big] \le \exp\big(2tE_{\pi_\tau}^2 [F]\big)E_{\pi_\tau} \big[\exp(2tF^2)\big] \le E_{\pi_\tau} \big[\exp(4tF^2)\big]. 
$$
Therefore, we obtain that 
$$
  E_{\pi_\tau} \left[\exp\left(\frac{c_-}{16c^2}F_*^2\right)\right] \le E_{\pi_\tau} \left[\exp\left(\frac{c_-}{4c^2}F^2\right)\right] \le \frac{2c_+}{c_-}\exp\left\{\frac{\tau^2}{2}\left(\frac2{c_-} - \frac1{c_+}\right)\right\}. 
$$
Using the $\phi_2$-condition (see Lemma \ref{lem:orlicz}), we can conclude that $F_*$ is a sub-Gaussian random variable of the order given by 
$$
  C(\tau, c, c_\pm) = \frac{64cc_+}{c_-^2}\exp\left\{\frac{\tau^2}{2}\left(\frac2{c_-} - \frac1{c_+}\right)\right\}. 
$$
The lemma then follows directly. 
\end{proof}

\section{Entropy estimate}
\label{sec:relative entropy}

In this section we prove Theorem \ref{thm:relative entropy} and Corollary \ref{cor:quantitative hl}. 
They are direct results of Lemma \ref{lem:main lemma} and the relative entropy inequality established in \cite{Yau91}. 

\begin{proof}[Proof of Theorem \ref{thm:relative entropy}]
Recall that $(\fp_i^n, \fr_i^n) = (\fp_n, \fr_n)(t, i/n)$ and $\bst_i^n = \bst_n(\fr_i^n)$. 
We start from Yau's entropy inequality stated in Appendix \ref{appendix:yau}: 
\begin{equation}
\label{yau 1}
  H'_n(t) \le -2n\gamma_nD\left(\sqrt{f_{t,n}}; \mu_{t,n}\right) + \beta\int f_{t,n}J_t^nd\mu_{t,n} + C\gamma_n, 
\end{equation}
where $C = C_{\beta,\fv,T}$. 
The remainder $J_t^n$ can be expressed by 
\begin{equation}
\label{yau 2}
  J_t^n = W_n(h_n) + E_t^n, \quad h_n = \partial_t\fr_n(t, \cdot), 
\end{equation}
where the functional $W_n$ is defined through \eqref{spatial variation}, and 
\begin{equation}
\label{yau 3}
  E_t^n = \sum_{i\in\bT_n} \epsilon_i^n \cdot \begin{pmatrix}p_i - \fp_i^n \\ V'_n(r_i) - \bst_i^n\end{pmatrix}, \quad \epsilon_i^n = -\partial_t\begin{pmatrix}\fp_n \\ \fr_n\end{pmatrix}\left(t, \frac in\right) + n\begin{pmatrix}\bst_{i+1}^n - \bst_i^n \\ \fp_i^n - \fp_{i-1}^n\end{pmatrix}. 
\end{equation}
For the integral of $E_t^n$, \eqref{entropy inequality 3} yields that 
$$
  \int f_{t,n}E_t^nd\mu_{t,n} \le H_n(t) + \sum_{i\in\bT_n} \log\int \exp\left\{\epsilon_i^n \cdot \begin{pmatrix}p_i - \fp_i^n \\ V'_n(r_i) - \bst_i^n\end{pmatrix}\right\}d\mu_{t,n}. 
$$
Notice that under $\mu_{t,n}$, $p_i - \fp_i^n$ is a Gaussian variable, while due to Lemma \ref{lem:subgaussian}, $V'_n(r_i) - \bst_i^n$ is sub-Gaussian of order $C = C_{\beta,\fv,T}$, so that 
\begin{equation}
\label{entropy estimate 1}
  \int f_{t,n}E_t^nd\mu_{t,n} \le H_n(t) + C\sum_{i\in\bT_n}|\epsilon_i^n|^2 \le H_n(t) + \frac{C'_{\beta,\fv,T}}n. 
\end{equation}
For $W_n(h_n)$, denote by $\mu_{t,n}^*$ the marginal distribution of $\mu_{t,n}$ on positions $(r_1, \ldots, r_n)$, and by $f_{t,n}^*$ the density of $(r_1, \ldots, r_n)(t)$ with respect to $\mu_{t,n}^*$. 
Applying Lemma \ref{lem:main lemma} with $\delta = 2/\beta$, and using the relation $H(f_{t,n}^*; \mu_{t,n}^*) \le H(f_{t,n}; \mu_{t,n})$ (see \eqref{entropy inequality 0}), 
\begin{align*}
%  \int f_{t,n}W_n(h_n)d\mu_{t,n}\le\ &\frac{2n\gamma_n}\beta D_{(r)}\left(\sqrt{f_{t,n}}; \mu_{t,n}\right) + C\big(H(f_{t,n}^*; \mu_{t,n}^*) + K_n\big) \\
  \int f_{t,n}W_n(h_n)d\mu_{t,n}\le\ &\frac{2n\gamma_n}\beta D\left(\sqrt{f_{t,n}}; \mu_{t,n}\right) + C\big(H(f_{t,n}^*; \mu_{t,n}^*) + \kappa_n\big) \\
  \le\ &\frac{2n\gamma_n}\beta D\left(\sqrt{f_{t,n}}; \mu_{t,n}\right) + C\big(H_n(t) + \kappa_n\big). 
\end{align*}
Hence, we obtain from \eqref{yau 1} that for all $t \in [0, T]$, 
$$
  H'_n(t) \le C_{\beta,\fv,T}\big(H_n(t) + \max\{\kappa_n, \gamma_n\}\big) = C_{\beta,\fv,T}(H_n(t) + K_n). 
$$
Theorem \ref{thm:relative entropy} then follows from the Gr\"onwall's inequality. 
\end{proof}

Corollary \ref{cor:quantitative hl} is a special case of the following result. 

\begin{cor}
\label{cor:general quantitative hl}
Let $F \in C(\bR)$ satisfy that $|F(r)| \le c|r|$. 
For all $p \in [1, 2)$, there is a constant $C = C(\beta, \fv, T, c, p)$, such that for all $h \in C(\bT)$, $t \in [0, T]$, 
$$
  \bE_n\left[\bigg|\frac1n\sum_{i\in\bT_n} h\left(\frac in\right)\big(F(r_i) - E_{\mu_{t,n}} [F(r_i)]\big)\bigg|^p\right] \le \frac{C(1 + H_n(t))^{\frac p2}\|h\|^p}{n^{\frac p2}}, 
$$
In particular, if \eqref{parameter 2} holds and $H_n(0) \le C_0\sqrt n$, then 
$$
  \bE_n \left[\bigg|\frac1n\sum_{i\in\bT_n} h\left(\frac in\right)\big(F(r_i) - E_{\mu_{t,n}} [F(r_i)] \big)\bigg|^p\right] \le \frac{C\|h\|^p}{n^{\frac p4}}, 
$$
with some constant $C = C(\beta, \fv, T, c, p, C_0)$. 
Similar result holds for $F(p_i)$. 
\end{cor}

\begin{proof}
Denote by $F_i = F(r_i) - E_{\mu_{t,n}} [F(r_i)]$ for short. 
By \eqref{entropy inequality 1}, 
$$
  \bP_n \left\{\left|\frac1n\sum_{i\in\bT_n} h_i^nF_i\right| > \lambda\right\} \le \frac{H_n(t) + \log2}{-\log\mu_{t,n}\big\{|\sum_i h_i^nF_i| > \lambda n\big\}}. 
$$
In view of Lemma \ref{lem:general subgaussian}, $\{F_i, i \in \bT_n\}$ is an independent family of sub-Gaussian variables of a uniform order under $\mu_{t,n}$. 
Then, with a constant $C = C_{\beta,\fv,T,c}$, 
$$
  E_{\mu_{t,n}} \left[\exp\left\{s\sum_{i\in\bT_n} h_i^nF_i\right\}\right] \le \exp\left\{\frac{Cs^2}2\sum_{i\in\bT_n} (h_i^n)^2\right\}, \quad \forall s \in \bR. 
$$
Therefore, $\sum_i h_i^nF_i$ is sub-Gaussian of order $Cn\|h\|^2$, and 
$$
  \mu_{t,n} \left\{\left|\sum_{i\in\bT_n} h_i^nF_i\right| > \lambda n\right\} \le 2\exp\left\{-\frac{\lambda^2n}{2C\|h\|^2}\right\}. 
$$
From this and the estimate above, we obtain that for any $\lambda > 0$, 
$$
  \bP_n \left\{\left|\frac1n\sum_{i\in\bT_n} h_i^nF_i\right| > \lambda\right\} \le \frac{C(1 + H_n(t))\|h\|^2}{\lambda^2n}. 
$$
By the moment estimate in Lemma \ref{lem:moment}, for all $p \in [1, 2)$, 
$$
  \bE_n \left[\bigg|\frac1n\sum_{i\in\bT_n} h_i^nF_i\bigg|^p\right] \le \frac{C(1 + H_n(t))^{\frac p2}\|h\|^p}{n^{\frac p2}}. 
$$
The second inequality in Corollary \ref{cor:general quantitative hl} follows directly from Theorem \ref{thm:relative entropy}. 
The parallel result for $F(p_i)$ can be proved in the same way. 
\end{proof}

\section{Boltzmann--Gibbs principle}
\label{sec:boltzmann gibbs}
In this section, we prove the proposition which is known in the literature as the \emph{Boltzmann--Gibbs principle}, firstly established for the equilibrium dynamics of zero range jump process in \cite{BR84}. 
It aims at determining the space-time variance of a local observation of conserved field by its linear approximation. 

\begin{prop}
\label{prop:boltzmann gibbs}
Suppose \eqref{parameter 2} and in additional 
\begin{equation}
\label{initial entropy}
  \lim_{n\to\infty} \frac{H_n(0)}{\sqrt n} = 0. 
\end{equation}
Let $g_n = g_n(t, x)$ be a sequence of functions on $[0, T] \times \bT$, such that $|g_n|_\bT$ and $|\partial_xg_n|_\bT$ are uniformly bounded for $t \in [0, T]$ and $n \ge 1$. 
For any $0 \le t < t' \le T$, 
\begin{equation}
  \lim_{n\to\infty} \bP_n \left\{\left|\frac1{\sqrt n} \int_t^{t'} \sum_{i\in\bT_n} g_i^n\Phi_i^nds\right| > \lambda\right\} = 0, \quad \forall \lambda > 0, 
\end{equation}
where $g_i^n = g_n(t, i/n)$ and $\Phi_i^n$ is given by \eqref{local function}. 
\end{prop}

We prove it along the approach in \cite[Theorem 5.1]{JM18b}. 

\begin{proof}
As we can consider $-g_n$ instead of $g_n$, it suffices to prove that 
$$
  \lim_{n\to\infty} \bP_n \left\{\int_t^{t'} \sum_{i\in\bT_n} g_i^n\Phi_i^nds > \lambda\sqrt n\right\} = 0. 
$$
Recall the auxiliary functional $W_{n,\delta}$ defined in Lemma \ref{lem:main lemma}, and the expressions $E_t^n$, $W_n(h_n)$ in \eqref{yau 2}. 
Define for any $\alpha > 0$ and $n \ge 1$ 
\begin{equation}
  U_{n,\alpha}(g_n) = W_{n,\frac1{2\alpha}}(g_n) + \frac\beta{2\alpha}\big[E_t^n + W_{n,\frac1\beta}(h_n)\big]. 
\end{equation}
Note that the parameter $\alpha$ is not needed here, but would be used in Section \ref{sec:tightness}. 
Let $\bP_{t,n}$ and $\bP_{t,n}^*$ be the law of the dynamics generated by $\cL_n$, respectively with initial distributions $\mu_{t,n}$ and $f_{t,n}d\mu_{t,n}$. 
By the Markov property, 
$$
  H\left(\frac{d\bP_{t,n}^*}{d\bP_{t,n}}; \bP_{t,n}\right) = H_n(f_{t,n}; \mu_{t,n}) = H_n(t). 
$$
Therefore, we can apply \eqref{entropy inequality 1} to the trajectory space to get 
\begin{align*}
  &\bP_n \left\{\int_t^{t'} \big[W_n(g) - U_{n,\alpha}(g)\big]ds > \lambda\sqrt n\right\} \\
  \le\ &\frac{H(t) + \log2}{-\log\bP_{t,n} \{\int_0^{t'-t} [W_n(g) - U_{n,\alpha}(g)]ds > \lambda\sqrt n\}}. 
\end{align*}
Applying \cite[Lemma 3.5]{JM18a} (see also \cite[Lemma A.2]{JM18b}) to the reference measures $\{\mu_{t+s,n}; s \in [0, t' - t]\}$, 
\begin{align*}
  &\log\bP_{t,n} \left\{\int_0^{t'-t} \big[W_n(g_n) - U_{n,\alpha}(g_n)\big]ds > \lambda\sqrt n\right\} \\
  \le &-\alpha\lambda\sqrt n + \int_0^{t'-t} \sup_f \bigg\{-n\gamma_n D\left(\sqrt f; \mu_{t+s,n}\right) + \\
  &\int f\left[\alpha\big(W_n(g_n) - U_{n,\alpha}(g_n)\big) + \frac{\beta J_{t+s}^n}2\right]d\mu_{t+s,n}\bigg\}ds, 
\end{align*}
where the supremum runs over all the density functions $f$ with respect to $\mu_{t+s,n}$. 
Since $J_t^n = E_t^n + W_n(h_n)$, \eqref{main lemma 1} in Lemma \ref{lem:main lemma} yields that 
$$
  \int f\left[\alpha\big(W_n(g_n) - U_{n,\alpha}(g_n)\big) + \frac{\beta J_{t+s}^n}2\right]d\mu_{t+s,n} \le n\gamma_n D\left(\sqrt f; \mu_{t+s,n}\right). 
$$
Hence, we can conclude that for all $\lambda > 0$, 
$$
  \log\bP_{t,n} \left\{\int_0^{t'-t} \big[W_n(g_n) - U_{n,\alpha}(g_n)\big]ds > \lambda\sqrt n\right\} \le -\alpha\lambda\sqrt n. 
$$
As the conditions and Theorem \ref{thm:relative entropy} assure that $H_n(t) = o(\sqrt n)$, 
\begin{equation}
\label{boltzmann gibbs 1}
  \lim_{n\to\infty} \bP_n \left\{\int_t^{t'} \big[W_n(g_n) - U_{n,\alpha}(g_n)\big]ds > \lambda\sqrt n\right\} \le \lim_{n\to\infty} \frac{H_n(t) + \log2}{\alpha\lambda\sqrt n} = 0. 
\end{equation}
For the integral of $U_{n,\alpha}$, notice that by Chebyshev's inequality, 
$$
  \bP_n \left\{\int_t^{t'} U_{n,\alpha}(g_n)ds > \lambda\sqrt n\right\} \le \frac1{\lambda\sqrt n}\bE_n \left[\left|\int_t^{t'} U_{n,\alpha}(g_n)ds\right|\right]. 
$$
By \eqref{main lemma 2} in Lemma \ref{lem:main lemma} and \eqref{entropy estimate 1}, with a constant $C = C_{\beta,\fv,T}$, 
$$
  \bE_n \big[|U_{n,\alpha}(g_n)|\big] \le C\big(\alpha + \alpha^{-1}\big)\big(1 + |g_n(t)|_\bT^2 + |\partial_xg_n(t)|_\bT\big)\big(H_n(t) + K_n\big). 
$$
From the conditions on $g_n$ and Theorem \ref{thm:relative entropy}, 
\begin{equation}
\label{boltzmann gibbs 2}
  \begin{aligned}
    &\lim_{n\to\infty} \bP_n \left\{\int_t^{t'} U_{n,\alpha}(g_n)ds > \lambda\sqrt n\right\} \\
    \le\ &\frac{C|t' - t|(1 + C_g)}\lambda\left(\alpha + \frac1\alpha\right)\lim_{n\to\infty} \frac{H_n(0) + K_n}{\sqrt n} = 0. 
  \end{aligned}
\end{equation}
where $C_g = \sup_{n\ge1,t\in[0,T]} \{|g_n(t)|_\bT^2 + |\partial_xg_n(t)|_\bT\}$. 
By summing up \eqref{boltzmann gibbs 1} and \eqref{boltzmann gibbs 2} together we prove the result. 
\end{proof}

\section{Convergence of finite-dimensional laws}
\label{sec:fdd}

In this section we prove that every possible weak limit point of $Y_t^n$ in \eqref{fluctuation field} satisfies \eqref{homogeneous p}. 
Let $H: [0, T] \times \bT \to \bR^2$ be a smooth function, and write $H = (H_1, H_2)$. 
By It\^o's formula, there is a square integrable martingale $M_t^n(H)$, such that 
$$
  Y_t^n(H(t)) - Y_0^n(H(0)) = \int_0^t \left(\frac d{ds} + \cL_n\right)Y_s^n(H(s))ds + M_t^n(H), 
$$
and the quadratic variation of $M_t^n(H)$ is given by 
\begin{equation}
\label{quadratic variation}
  \big\langle M_t^n(H) \big\rangle = n\gamma_n\int_0^t \Gamma_n \big[Y_t^n(H(s))\big]ds, \quad %\Gamma_nf = \big(\cS_n + \overline\cS_n)[f^2] - 2f\big(\cS_n + \overline\cS_n\big)f. 
  \Gamma_nf = \cS_n[f^2] - 2f\cS_nf. 
\end{equation}
Recall that $\fu_n = (\fp_n, \fr_n)$ denotes the solution to \eqref{quasi linear p} with $\sigma = \sigma_n$, $(\fp_i^n, \fr_i^n) = \fu_n(s, i/n)$ and $\bst_i^n = \bst_n(\fr_i^n)$. 
We also write $\fb_i^n = \bst'_n(\fr_i^n)$. 
Through direct computation, 
%\begin{align*}
%  &\frac d{ds}Y_s^n(H(s)) = Y_s^n(\partial_sH) - \frac1{\sqrt n}\sum_{i\in\bT_n} H\left(s, \frac in\right) \cdot \partial_s\fu_n\left(s, \frac in\right) \\
%  =\ &{\color{red}\frac1{\sqrt n}\sum \left[\partial_sH - \begin{bmatrix}0 &1 \\ \fb_i^n &0\end{bmatrix}\begin{pmatrix}\nabla_{i-1}^nH_1 \\ \nabla_i^nH_2\end{pmatrix}\right] \cdot \big(\eta_i - \fu_i^n\big)} + \frac1{\sqrt n}\sum \begin{pmatrix}\nabla_{i-1}^nH_1 \\ \nabla_i^nH_2\end{pmatrix} \cdot \begin{pmatrix}\fb_i^n(r_i - \fr_i^n) \\ p_i - \fp_i^n\end{pmatrix} \\
%  &{\color{red}+\ \frac1{\sqrt n}\sum H\left(s, \frac in\right) \cdot \left[-\partial_s\fu_n\left(s, \frac in\right) + n\begin{pmatrix}\bst_{i+1}^n - \bst_i^n \\ \fp_i^n - \fp_{i-1}^n\end{pmatrix}\right]} + \frac1{\sqrt n}\sum \begin{pmatrix}\nabla_{i-1}^nH_1 \\ \nabla_i^nH_2\end{pmatrix} \cdot \begin{pmatrix}\bst_i^n \\ \fp_i^n\end{pmatrix}, 
%\end{align*}
\begin{align*}
  \frac d{ds}Y_s^n(H) =\ &\frac1{\sqrt n}\sum_{i\in\bT_n} \big(\partial_sH_i^n - L_i^nH\big) \cdot \big(\eta_i - \fu_i^n\big) \\
  &- \frac1{\sqrt n}\sum_{i\in\bT_n} H_i^n \cdot \big(\partial_s\fu_i^n - \Lambda_i^n\big) \\
  &+ \frac1{\sqrt n}\sum_{i\in\bT_n} \begin{pmatrix}\nabla_{i-1}^nH_1 \\ \nabla_i^nH_2\end{pmatrix} \cdot \begin{pmatrix}\bst_i^n + \fb_i^n(r_i - \fr_i^n) \\ p_i\end{pmatrix}, \\
  \cA_n[Y_s^n(H)] =\ &n^{-\frac32}\sum_{i\in\bT_n} \begin{pmatrix}\nabla_{i-1}^nH_1 \\ \nabla_i^nH_2\end{pmatrix} \cdot \begin{pmatrix}-V_n'(r_i) \\ -p_i\end{pmatrix}, \\
%  \big(\cS_n + \overline\cS_n\big)[Y_s^n(H)] =\ &2^{-1}n^{-\frac52}\sum_{i\in\bT_n} \Delta_i^nH \cdot \begin{pmatrix}p_i \\ V'_n(r_i)\end{pmatrix}. 
  \cS_n[Y_s^n(H)] =\ &2^{-1}n^{-\frac52}\sum_{i\in\bT_n} \Delta_i^nH_2V'_n(r_i). 
\end{align*}
Here $\nabla_i^n$ and $\Delta_i^n$ are discrete derivatives given by 
$$
  \nabla_i^nf = n\left[f\left(\frac{i+1}n\right) - f\left(\frac in\right)\right], \quad \Delta_i^nf = n\big(\nabla_i^nf - \nabla_{i-1}^nf\big), 
$$
while the operator $L_i^n = L_i^n(s)$ and approximate field $\Lambda_i^n = \Lambda_i^n(s)$ are 
$$
  L_i^nH = \begin{bmatrix}0 &1 \\ \fb_i^n &0\end{bmatrix}\begin{pmatrix}\nabla_{i-1}^nH_1 \\ \nabla_i^nH_2\end{pmatrix}, \quad \Lambda_i^n = n\begin{pmatrix}\bst_{i+1}^n - \bst_i^n \\ \fp_i^n - \fp_{i-1}^n\end{pmatrix}. 
$$
With the notations above, $Y_t^n(H(t))$ is split into 
\begin{equation}
\label{ito}
  Y_t^n(H(t)) = Y_0^n(H(0)) + \sR_t^n(H) + \sA_t^n(H) + \sS_t^n(H) + \sW_t^n(H) + M_t^n(H), 
\end{equation}
where $\sR_t^n$, $\sA_t^n$, $\sS_t^n$ and $\sW_t^n$ are given respectively by 
\begin{align*}
  &\sR_t^n(H) = \frac1{\sqrt n}\int_0^t \sum_{i\in\bT_n} \left[\partial_sH\left(s, \frac in\right) - L_i^nH(s)\right] \cdot \big(\eta_i - \fu_i^n\big)ds, \\
  &\sA_t^n(H) = \frac1{\sqrt n}\int_0^t \sum_{i\in\bT_n} H\left(s, \frac in\right) \cdot \left[-\partial_s\fu_n\left(s, \frac in\right) + \Lambda_i^n\right]ds, \\
%  &\sS_t^n(H) = \frac{\gamma_n}{2\sqrt n}\frac1n\int_0^t \sum_{i\in\bT_n} \Delta_i^nH(s) \cdot \begin{pmatrix}p_i \\ V'_n(r_i)\end{pmatrix}ds, \\
  &\sS_t^n(H) = \frac{\gamma_n}{2\sqrt n}\frac1n\int_0^t \sum_{i\in\bT_n} \Delta_i^nH_2(s)V'_n(r_i)ds, \\
  &\sW_t^n(H) = \frac1{\sqrt n}\int_0^t \sum_{i\in\bT_n} \nabla_{i-1}^nH_1(s)\big[-V'_n(r_i) + \bst_i^n + \fb_i^n(r_i - \fr_i^n)\big]ds. 
\end{align*}
The finite-dimensional convergence is stated below. 

\begin{prop}
Assume \eqref{parameter 2} and \eqref{initial entropy}. 
Define $h_n = h_n(t, x)$ be the solution to the following adjoint equation on $(t, x) \in [0, T] \times \bT$: 
\begin{equation}
\label{backward 1}
  \partial_th_n - \begin{bmatrix}0 &1 \\ \bst'_n(\fr_n) &0\end{bmatrix}\partial_xh_n = 0, \quad h_n(0, \cdot) = H, 
\end{equation}
with some fixed initial condition $H \in C^{\infty}(\bT)$. 
For any $\lambda > 0$, 
$$
  \lim_{n\to\infty} \bP_n\left\{\sup_{t\in[0,T]} |Y_t^n(h_n(t)) - Y_0^n(H)| > \lambda\right\} = 0, \quad \forall \lambda > 0. 
$$
\end{prop}

\begin{proof}
We investigate each term in \eqref{ito} respectively. 
The martingale term $M_t^n$ is the easiest. 
From \eqref{quadratic variation}, for all $t \in [0, T]$, 
\begin{equation}
\label{fdd m}
  \bE_n \big[|M_t^n(h_n)|^2\big] = \gamma_n\int_0^t \sum_{i\in\bT_n} \frac1{n^2}\big(\nabla_i^n h_n\big)^2ds \le \frac{\gamma_nt}n\sup_{s\in[0,t]} |\partial_xh_n(s)|_\bT^2. 
\end{equation}
From Doob's inequality, 
$$
  \bE_n \left[\sup_{t\in[0,T]} |M_t^n(h_n)|^2\right] \le 4\bE_n \big[|M_T^n(h_n)|^2\big] \le \frac{C\gamma_n}n, 
$$
which vanishes as $n \to \infty$. 
For the integral $\sR_t^n$, by \eqref{backward 1}, 
$$
  \left|\partial_sh_n\left(s, \frac in\right) - L_i^nh_n(s)\right| \le \frac{C|\partial_x^2h_n(s)|_\bT}n. 
$$
Using this estimate and Corollary \ref{cor:quantitative hl}, for all $p \in [1, 2)$, 
\begin{equation}
\label{fdd r}
  \bE_n \left[\bigg|\frac1{\sqrt n}\sum_{i\in\bT_n} \left[\partial_sh_n\left(s, \frac in\right) - L_i^nh_n(s)\right] \cdot \begin{pmatrix}p_i - \fp_i^n \\ r_i - \fr_i^n\end{pmatrix}\bigg|^p\right] \le \frac{C_p|\partial_x^2h_n(s)|_\bT^p}{n^{p(1+\frac14-\frac12)}}. 
\end{equation}
Taking $p = 1$ and with Cauchy--Schwarz inequality, 
$$
  \lim_{n\to\infty} \bE_n \left[\sup_{t\in[0,T]} \big|\sR_t^n(h_n)\big|\right] \le \lim_{n\to\infty} CTn^{-\frac34} = 0. 
$$
For the integral $\sA_t^n$, since we assume that the quasi-linear system \eqref{quasi linear p} has smooth solution at least up to time $T$, therefore 
$$
  \sup_{s\in[0,T]} \left|\partial_t\fu_n\left(s, \frac in\right) - \Lambda_i^n(s)\right| \le \frac{C_{\beta,\fv,T}}n. 
$$
This gives us the uniform estimate that 
\begin{equation}
\label{fdd a}
  \left|\frac1{\sqrt n}\sum_{i\in\bT_n} h_n\left(s, \frac in\right)\left[\partial_s\fu_n\left(s, \frac in\right) - \Lambda_i^n\right]\right| \le \frac{C|h_n(s)|_\bT}{\sqrt n}. 
\end{equation}
Hence, $|\sA_t^n|$ vanishes uniformly on $t \in [0, T]$ as $n \to \infty$. 
For the integral $\sS_t^n$, observe that the integrand can be bounded from above by 
%$$
%  \left|\sum_{i\in\bT_n} \Delta_i^nh_n(s) \cdot \begin{pmatrix}p_i - \fp_i^n \\ V'_n(r_i) - \bst_i^n\end{pmatrix}\right| + n|\partial_x^2h_n(s)|_\bT\left|\begin{pmatrix}\fp_n \\ \bst_n(\fr_n)\end{pmatrix}\right|_\bT. 
%$$
$$
    \left|\sum_{i\in\bT_n} \Delta_i^nh_{n,2}(s)\big(V'_n(r_i) - \bst_i^n\big)\right| + n\big|\partial_x^2h_{n,2}(s)\bst_n(\fr_n)\big|_\bT. 
$$
Again, by Corollary \ref{cor:quantitative hl}, for all $p \in [1, 2)$, 
\begin{equation}
\label{fdd s}
%  \bE_n \left[\bigg|\frac1n\sum_{i\in\bT_n} \Delta_i^nh_n(s) \cdot \begin{pmatrix}p_i \\ V'_n(r_i)\end{pmatrix}\bigg|^p\right] \le C_p|\partial_x^2h_n(s)|_\bT^p\left(n^{-\frac p4} + 1\right). 
  \bE_n \left[\bigg|\frac1n\sum_{i\in\bT_n} \Delta_i^nh_{n,2}(s)V'_n(r_i)\bigg|^p\right] \le C_p|\partial_x^2h_n(s)|_\bT^p\left(n^{-\frac p4} + 1\right). 
\end{equation}
Taking $p = 1$ and using Cauchy--Schwarz inequality, 
$$
  \lim_{n\to\infty} \bE_n \left[\sup_{t\in[0,T]} \big|\sS_i^n(h_n)\big|\right] \le \lim_{n\to\infty} \frac{C\gamma_n}{2\sqrt n} = 0. 
$$
Finally, we apply Proposition \ref{prop:boltzmann gibbs} to $\sW_t^n$ to get that 
$$
  \lim_{n\to\infty} \bP_n \left\{\sup_{t\in[0,T]} \big|\sW_t^n(h_n)\big| > \lambda\right\} = 0, \quad \forall \lambda > 0. 
$$
The proof is then completed. 
\end{proof}

\section{Tightness}
\label{sec:tightness}

In this section, we prove that the laws of $\{Y_t^n; t \in [0, T]\}$ forms a tight sequence in proper trajectory space. 
We start with two lemmas. 
Suppose that for $n \ge 1$ and $f \in C^2(\bT)$, $\{X_t^n = X_t^n(f); t \in [0, T]\}$ is a random field on $\Omega_n$. 
Define 
$$
  \sX_t^n(f) = \int_0^t X_s^n(f)ds, \quad \forall t \in [0, T]. 
$$
By Kolmogorov--Prokhorov's tightness criterion, to show the tightness of $\sX_\cdot^n$ on the $\alpha$-H\"older continuous path space $C^\alpha([0, T]; \cH_{-k})$, one need to estimate 
$$
  \big\|\sX_{t'}^n - \sX_t^n\big\|_{-k}^2 = \sum_{m\in\bZ} \frac1{(1 + m^2)^k}\big|\sX_{t'}^n(\varphi_m) - \sX_t^n(\varphi_m)\big|^2, 
$$
with the Fourier basis $\varphi_m$ defined in Section \ref{sec:introduction}. 
Since Corollary \ref{cor:quantitative hl} and \ref{cor:general quantitative hl} only hold with powers $p < 2$, the next result is helpful here. 

\begin{lem}
\label{lem:tightness 1}
Assume some $p > 1$ and $a > 0$, such that 
$$
  \bE_n \big[|X_t^n(\varphi_m)|^p\big] \le C|m|^{ap}, \quad \forall m \in \bZ,\ t \in [0, T]. 
$$
Then there exists a constant $C_p$, such that for $k > a + 3/2$, 
$$
  \bP_n \left\{\big\|\sX_{t'}^n - \sX_t^n\big\|_{-k} > \lambda\right\} \le \frac{C_p|t' - t|^p}{\lambda^p}, \quad \forall 0 \le t < t' \le T. 
$$
In particular, $\sX_\cdot^n$ is tight in $C^\alpha([0, T]; \cH_{-k})$ for $\alpha < 1 - 1/p$. 
\end{lem}

The proof of Lemma \ref{lem:tightness 1} is direct and we postpone it to the end of this section. 
In order to use Lemma \ref{lem:tightness 1}, we need the following priori moment estimate. 

\begin{lem}
\label{lem:tightness 2}
Assume \eqref{parameter 2} and \eqref{parameter tem} with some $\epsilon \in \bR$. 
For all $1 \le p < p_\epsilon = (4 - 2\epsilon)/(3 - 2\epsilon) \lor 2$ and $f \in C^2(\bT)$, 
$$
  \bE_n \big[|Y_t^n(f)|^p\big] \le C\big(1 + |f|_\bT^p + |f'|_\bT^{\delta p} + |f''|_\bT^p\big). 
$$
where $\delta = \delta_\epsilon = (3 - 2\epsilon)/(2 - \epsilon) \land 1$ and $C = C_{\beta,\fv,T,\epsilon,p}$. 
\end{lem}

Note that if we apply Corollary \ref{cor:quantitative hl} to the left-hand side above, the upper bound could diverse. 
The additional condition \eqref{parameter tem} helps to avoid this. 

\begin{proof}[Proof of Lemma \ref{lem:tightness 2}]
Fix some $t \in [0, T]$, and define $f_n = f_n(s, x)$ to be the solution of the following \emph{backward equation} on $(s, x) \in [0, t] \times \bT$: 
\begin{equation}
\label{backward 2}
  \partial_sf_n - \begin{bmatrix}0 &1 \\ \bst'_n(\fr_n) &0\end{bmatrix}\partial_xf_n = 0, \quad f_n(t, \cdot) = f. 
\end{equation}
Apply \eqref{ito} with $H = f_n$ and estimate each term in the right-hand side. 

For $Y_0^n(f_n(0))$, use \eqref{entropy inequality 1} to get that for all $\lambda > 0$, 
$$
  \bP_n \big\{|Y_0^n(f_n(0))| > \lambda\big\} \le \frac{H_n(0) + \log2}{-\log\mu_{0,n}\{|Y_0^n(f_n(0))| > \lambda\}}. 
$$
By Lemma \ref{lem:general subgaussian} and the independence, with some $C = C_{\beta,\fv}$, 
$$
  \mu_{0,n} \big\{|Y_0^n(f_n(0))| > \lambda\big\} \le 2\exp\left\{-\frac{\lambda^2}{2C\|f_n(0)\|^2}\right\}. 
$$
As $H_n(0)$ is assumed to be bounded, 
$$
  \bP_n \big\{|Y_0^n(f_n(0))| > \lambda\big\} \le \frac{C\|f_n(0)\|^2(H_n(0) + 1)}{\lambda^2} \le \frac{C'\|f_n(0)\|^2}{\lambda^2}. 
$$
Using Lemma \ref{lem:moment}, for all $p \in [1, 2)$, 
$$
  \bE_n \big[|Y_0^n(f_n(0))|^p\big] \le C_p\|f_n(0)\|^p \le C'_p|f|_\bT^p. 
$$
For $M_t^n$, apply an interpolation of \eqref{fdd m} with $p \in [1, 2]$: 
$$
  \bE_n \big[|M_t^n(f_n)|^p\big] \le \big(\gamma_ntn^{-1}\big)^{\frac p2}\sup_{s\in[0,t]}|\partial_xf_n(s)|_\bT^p \le C_{T,p}|f'|_\bT^p. 
$$
For $\sR_t^n$, it is easy to obtain from \eqref{fdd r} that, for $p \in [1, 2)$, 
$$
  \bE_n \left[\big|\sR_t^n(f_n)\big|^p\right] \le t^{p-1}\int_0^t Cn^{-\frac{3p}4}\|\partial_x^2f_n(s)\|_\bT^pds \le C'|f''|_\bT^p. 
$$
For $\sA_t^n$, the upper bound of $p$-moment follows directly from the uniform estimate in \eqref{fdd a}. 
For $\sS_t^n$, the estimate can be obtained from \eqref{fdd s} similarly to $\sR_t^n$. 
%$$
%  \bE_n \left[\big|\sS_t^n(f_n)\big|^p\right] \le \left(\frac{\gamma_n}{2\sqrt n}\right)^pt^{p-1}\int_0^t C_p\big|\partial_x^2f_n(s)\big|_\bT^pds \le C_{T,p}|f''|_\bT^p. 
%$$

The only term needs extra effort is $\sW_t^n$. 
Rewrite this term as 
\begin{align*}
  \sW_t^n(f) =\ &\frac{\sigma_n}{\sqrt n}\int_0^t \sum_{i\in\bT_n} \left(-U'(r_i) + \frac{\bst_i^n - \fr_i^n}{\sigma_n}\right)\nabla_{i-1}^nfds \\
  &+ \frac{\fb_i^n - 1}{\sqrt n}\int_0^t \sum_{i\in\bT_n} (r_i - \fr_i^n)\nabla_{i-1}^nfds. 
\end{align*}
By \eqref{asymptotic tension}, $|\fb_i^n - 1| = O(\sigma_n)$, so we obtain from Theorem \ref{thm:relative entropy} and Corollary \ref{cor:general quantitative hl} that 
$$
  \bE_n \left[\big|\sW_t^n(f)\big|^q\right] \le C(q)t^q\sigma_n^q(H_n(0) + K_n + 1)^{\frac q2}|f'|_\bT^q. 
$$
for all $q \in [1, 2)$ with some $C(q) = C_{\beta,\fv,T}(q)$. 
Thus, for all $\lambda > 0$, 
\begin{equation}
\label{boundedness 1}
  \bP_n\left\{\big|\sW_t^n(f)\big| > \lambda\right\} \le \frac{Ct^q\sigma_n^q(H_n(0) + K_n + 1)^{\frac q2}|f'|_\bT^q}{\lambda^q}. 
\end{equation}
In view of Lemma \ref{lem:moment}, if $\sigma_n^2K_n$ is bounded, or equivalently $\epsilon \ge 1$, 
$$
  \bE_n \left[\big|\sW_t^n(f)\big|^p\right] \le C_{\beta,\fv,T,p}t^p|f'|_\bT^p 
$$
for all $p \in [1, 2)$ and we obtain the desired estimate. 
On the other hand, using \eqref{boltzmann gibbs 1} and \eqref{boltzmann gibbs 2} with $\alpha = t^{-1/2}$, we get the same probability bounded by 
$$
  \bP_n\left\{\big|\sW_t^n(f)\big| > \lambda\right\} \le \frac{C\sqrt t(H_n(0) + K_n + 1)}{\lambda\sqrt n}\big(1 + |f'|_\bT^2 + |f''(s)|_\bT\big). 
$$
Note that the expression above vanishes for large $n$. 
Therefore, in case that $0 < \epsilon < 1$, we can apply the following interpolation for $\theta \in (0, 1)$ that 
\begin{align*}
  \bP_n\left\{\big|\sW_t^n(f)\big| > \lambda\right\} \le\ &\frac{C(q, \theta)M_f(q, \theta)}{\lambda^{q\theta + 1 - \theta}}t^{q\theta + \frac{1-\theta}2}\ \times \\
  &\sigma_n^{q\theta}\big(H_n(0) + K_n + 1\big)^{\frac{q\theta}2+1-\theta}n^{\frac{\theta-1}2}, 
\end{align*}
where $C(q, \theta) = C_{\beta,\fv,T}(q, \theta)$ and 
$$
  M_f(q, \theta) = |f'|_\bT^{q\theta}\big(1 + |f'|_\bT^2 + |f''|_\bT\big)^{1-\theta} \le C'(q, \theta)\big(1 + |f'|_\bT^{q\theta+2(1-\theta)} + |f''|_\bT\big). 
$$
To assure that the second line above is bounded in $n$, choose 
$$
  \theta = \theta(\epsilon, q) = \frac1{1 + (1 - \epsilon)q}. 
$$
The estimate above becomes 
\begin{equation}
\label{tightness w}
  \bP_n\left\{\big|\sW_t^n(f)\big| > \lambda\right\} \le C(\epsilon, q)\lambda^{-q'}\big(1 + |f'|_\bT^{q_*} + |f''|_\bT\big)t^{q_{*\!*}}, 
\end{equation}
where 
$$
  q' = \frac{(2 - \epsilon)q}{1 + (1 - \epsilon)q}, \quad q_* = \frac{(3 - 2\epsilon)q}{1 + (1 - \epsilon)q}, \quad q_{*\!*} = \frac{(3 - \epsilon)q}{2 + 2(1 - \epsilon)q}. 
$$
%\begin{align*}
%  &q' = q\theta + 1 - \theta = \frac{(2 - \epsilon)q}{1 + (1 - \epsilon)q}, \\
%  &q_* = q\theta + 2(1 - \theta) = \frac{(3 - 2\epsilon)q}{1 + (1 - \epsilon)q}, \\
%  &q_{*\!*} = q\theta + \frac{1 - \theta}2 = \frac{(3 - \epsilon)q}{2 + 2(1 - \epsilon)q}. 
%\end{align*}
The dependence on $t$ is not important here. 
Notice that $1 < q' \le 2$ for $q \in [1, 2)$ and $\epsilon < 1$, thus we get from Lemma \ref{lem:moment} that for all $p \in [1, q')$, 
$$
  \bE_n \left[\big|\sW_t^n(f)\big|^p\right] \le C_{\beta,\fv,T}(p, q')\big(1 + |f'|_\bT^{\delta p} + |f''|_\bT^{p/q'}\big), 
$$
where $\delta = q_*/q' = (3 - 2\epsilon)/(2 - \epsilon)$ is independent of $q$. 
%With $q$ grows from $1$ to $2$, 
%$$
%  q' = 1 \to \frac{4 - 2\epsilon}{3 - 2\epsilon}, \quad q_* = \frac{3 - \epsilon}{4 - 2\epsilon} \to \frac{3 - \epsilon}{3 - 2\epsilon}, \quad q_{*\!*} = \frac{3 - 2\epsilon}{2 - 2\epsilon} \to 2. 
%$$
Since $q$ can be taken arbitrarily close to $2$, the inequality above holds for all $1 < p < p_\epsilon$. 
%$$  q \downarrow \frac2{1 + \epsilon} < 2, \quad q_* \downarrow 1, \quad q_{*\!*} \downarrow \frac{6 - 4\epsilon}{3 - \epsilon} < 2, \quad q' \downarrow \frac{4 - 2\epsilon}{3 - \epsilon} > 1$$
Finally, the lemma is proved by collecting all the moment estimate together. 
\end{proof}

With these lemmas, we can prove the tightness of $Y_\cdot^n$ stated below. 

\begin{prop}
Assume \eqref{parameter 2} and \eqref{parameter tem} with some $\epsilon > 0$. 
The laws of $\{Y_t^n; t \in [0, T]\}$ is tight with respect to the topology of $C([0, T]; \cH_{-k})$ for $k > 9/2$. 
\end{prop}

\begin{proof}
We need to investigate the tightness for each term in \eqref{ito}. 
Similar with \eqref{fdd m}, it is easy to observe that $M_t^n$ is tight on $C([0, T]; \cH_{-k})$ for $k > 3/2$: 
$$
  \bE_n \big[\|M_{t'}^n - M_t^n\|_{-k}^2\big] \le \frac{\gamma_n|t' - t|}n\sum_{m\in\bZ} \frac{|\varphi'_m|_\bT^2}{(1 + m^2)^k} \to 0. 
$$
The computations for $\sR_\cdot^n$, $\sA_\cdot^n$ and $\sS_\cdot^n$ are also direct. 
For $\sR_t^n$, note that 
$$
  \sR_t^n(\varphi_m) = -\int_0^t Y_s^n(L_i^n\varphi_m)ds. 
$$
From Lemma \ref{lem:tightness 2}, for $1 \le p < p_\epsilon$, 
$$
  E \big[|Y_t^n(L_i^n\varphi_m)|^p\big] \le C|m|^{3p}. 
$$
By Lemma \ref{lem:tightness 1}, $\sR_\cdot^n$ is tight on $C^\alpha([0, T]; \cH_{-k})$ for $k > 9/2$, $\alpha < \alpha_\epsilon = 1 - 1/p_\epsilon$. 
For $\sA_t^n$, observe that from \eqref{fdd a}, for $k > 1/2$, 
$$
  \big\|\sA_{t'}^n - \sA_t^n\big\|_{-k}^2 \le \frac{C|t' - t|^2}n\sum_{m\in\bZ} \frac{|\varphi_m|_\bT^2}{(1 + m^2)^k} \to 0. 
$$
Therefore, it is tight in $C^1([0, T]; \cH_{-k}(\bT))$ for $k > 1/2$. 
For $\sS_t^n$, notice that $\varphi''_m = Cm^2\varphi_m$. 
Substituting this into \eqref{fdd s}, we obtain that 
$$
%  \bE_n \left[\bigg|\frac1n\sum_{i\in\bT_n} \Delta_i^n\varphi_m \cdot \begin{pmatrix}p_i \\ V'_n(r_i)\end{pmatrix}\bigg|^p\right] \le Cm^{2p}, \quad \forall p \in [1, 2). 
  \bE_n \left[\bigg|\frac1n\sum_{i\in\bT_n} \Delta_i^n\varphi_mV'_n(r_i)\bigg|^p\right] \le C|m|^{2p}, \quad \forall p \in [1, 2). 
$$
By Lemma \ref{lem:tightness 1}, it is tight on $C^\alpha([0, T]; \cH_{-k}(\bT))$ for $k > 7/2$, $\alpha < 1/2$. 

We are left with $\sW_t^n$. 
In order to prove its tightness, we need to track the power of $t$ in \eqref{tightness w}. 
Repeat the computation, we obtain that  for $1 \le p < p_\epsilon$, 
$$
  \bE_n \left[\big|\sW_{t'}^n(f) - \sW_t^n(f)\big|^p\right] \le C\big(1 + |f'|_\bT^{\delta p} + |f''|_\bT\big)|t' - t|^{q_{*\!*}p/p_\epsilon}. 
$$
As $\epsilon > 0$, $q_{**} > 1$ when $2/(1 + \epsilon) < q < 2$. 
Therefore, there exists some $p > 1$, smaller than but close to $p_\epsilon$, such that 
$$
  \bE_n \left[\big|\sW_{t'}^n(f) - \sW_t^n(f)\big|^p\right] \le C\big(1 + |f'|_\bT^{\delta p} + |f''|_\bT\big)|t' - t|^{p'}, 
$$
where $p' > 1$. 
Applying the estimate to $f = \varphi_m$ and noticing that $\delta < 3/2$, by Lemma \ref{lem:tightness 1} we know that $\sW_t^n$ is tight in $C^\alpha([0, T], \cH_{-k})$ for $\alpha < 1 - 1/p$ and $k > 9/2$. 
In conclusion, the laws of $Y_\cdot^n$ is tight with respect to the topology of $C([0, T]; \cH_{-k})$ with $k > 9/2$. 
\end{proof}

\begin{proof}[Proof of Lemma \ref{lem:tightness 1}]
For any $t$, $t' \in [0, T]$, 
$$
  \bE_n \left[\big|\sX_{t'}^n(\varphi_m) - \sX_t^n(\varphi_m)\big|^p\right] \le C|t' - t|^p|m|^{ap}. 
$$
For any $\epsilon > 0$, with $C(\epsilon) = \sum_{m\in\bZ} (1 + m^2)^{-\frac12-\epsilon}$, 
\begin{align*}
  &\bP_n \left\{\big\|\sX_{t'}^n - \sX_t^n\big\|_{-k} > \lambda\right\} \\
  \le\ &\sum_{m\in\bZ} \bP_n \left\{\big|\sX_{t'}^n(\varphi_m) - \sX_t^n(\varphi_m)\big| \ge \frac{\lambda}{\sqrt{C(\epsilon)}}(1 + m^2)^{\frac 12(k-\frac12-\epsilon)}\right\} \\
%  \le\ &\frac{C(p, \epsilon)}{\lambda^p}\sum_{m\in\bZ} (1 + m^2)^{-\frac p2(k-\frac12-\epsilon)}\bE_n \left[\big|\sX_{t'}^n(\varphi_m) - \sX_t^n(\varphi_m)\big|^p\right] \\
  \le\ &\frac{C(p, \epsilon)}{\lambda^p}\sum_{m\in\bZ} (1 + m^2)^{-\frac p2(k-\frac12-\epsilon)}C|t' - t|^p|m|^{ap}. 
\end{align*}
Hence, for any $k > a + 3/2$, the probability is bounded from above by 
$$
   \frac{C'(p, \epsilon)|t' - t|^p}{\lambda^p} \sum_{m\in\bZ} \frac1{(1 + m^2)^{\frac p2(1-\epsilon)}}. 
$$
By fixing some $\epsilon$ such that $p(1 - \epsilon) > 1$, we obtain the desired estimate. 
For the tightness, only note that by Lemma \ref{lem:moment}, 
$$
  \bE_n \left[\big\|\sX_{t'}^n - \sX_t^n\big\|_{-k}^q\right] \le C_{p,q}|t' - t|^q, \quad \forall q \in (1, p), 
$$
and invoke Kolmogorov--Prokhorov's tightness criterion. 
\end{proof}

\section{Equivalence of ensembles}
\label{sec:ee}

In this section we prove the equivalence of ensembles for inhomogeneous canonical measure, which is used in Section \ref{sec:lemma}. 
Our main result, Proposition \ref{prop:ee exponential}, is valid not only for the weakly anharmonic case, but also for the general anharmonic case. 

Recall that for $\tau \in \bR$, $\sigma \in [0, 1)$, we have the probability measure 
$$
  \pi_{\tau,\sigma}(dr) = \exp\left\{-\frac{\beta r^2}2 - \beta\sigma U(r) + \beta\tau r - \beta G_\sigma(\tau)\right\}dr. 
$$
For simplicity, we fix $\beta = 1$ in this section, but the arguments apply to any fixed $\beta$ naturally. 
For $\vt = (\tau_1, \ldots, \tau_n)$, define $\mu_{\vt,\sigma}$ as the product measure $\otimes_{j=1}^n \pi_{\tau_j,\sigma}(dr_j)$ on $\bR^n$. 
For bounded continuous function $F$ on $\bR^n$, define 
$$
  \langle F|u \rangle_{\vt,\sigma} = E_{\mu_{\vt,\sigma}} \big[F|r_{(n)} = u\big], \quad r_{(n)} = \frac1n\sum_{j=1}^n r_j. 
$$
The conditioned probability distribution $\langle\;\cdot\;|u \rangle_{\vt,\sigma}$ is called the \emph{micro canonical ensemble}, while $\mu_{\vt,\sigma}$ is called the \emph{canonical ensemble}. 

First of all, we present a basic property of the micro canonical ensemble, which would be frequently used hereafter in this section. 
Note that as $U$ is smooth, we can define the regular conditional expectation $\langle F|u \rangle_{\vt,\sigma}$ point-wisely for all $u \in \bR$. 

\begin{prop}
For all $u \in \bR$, $\vt \in \bR^n$ and $\tau \in \bR$, 
\begin{equation}
\label{dependence on parameter}
  \langle F|u \rangle_{\vt,\sigma} = \langle F|u \rangle_{\vt-\tau,\sigma}, 
\end{equation}
where $\vt - \tau \triangleq (\tau_1 - \tau, \ldots, \tau_n - \tau)$. 
Moreover, there is $\vn = \vn(u; \vt, \sigma)$, such that 
\begin{equation}
\label{tension vector function}
  E_{\mu_{\vn,\sigma}} [r_{(n)}] = u, \quad \langle \;\cdot\;|u \rangle_{\vn,\sigma} = \langle \;\cdot\;|u \rangle_{\vt,\sigma}. 
\end{equation}
In particular when $n = 1$, $\nu(u; \tau, \sigma) = \bst_\sigma(u)$. 
\end{prop}

\begin{proof}
By direct computation, for $F = F(r_1, \ldots, r_k)$ and $n \ge k$, 
\begin{equation}
\label{micro canonical expectation}
  \langle F|u \rangle_{\vt,\sigma} = \frac{n}{n - k}\int_{\bR^k} \frac1{f_{\vt,\sigma}(u)}f_{\vt^*,\sigma}\left(\frac{nu - kr_{(k)}}{n - k}\right)F(r_1, \dots, r_k) \prod_{j=1}^k \pi_{n,j}(dr_j), 
\end{equation}
where $f_{\vt,\sigma}$ denotes the density of $r_{(n)}$ under $\mu_{\vt,\sigma}$ and $\vt^* = (\tau_{k+1}, \ldots, \tau_n)$. 
Observe that for bounded continuous function $h$ on $\bR$ and $\tau \in \bR$, 
\begin{align*}
  E_{\mu_{\vt,\sigma}} \big[h \circ r_{(n)}\big] &= \int h\left(\frac1n\sum_{j=1}^n r_j\right)\exp\left\{\sum_{j=1}^n \tau_jr_j - \frac{r_j^2}2 - \sigma U(r_j) - G_\sigma(\tau_j)\right\}d\bfr \\
  &= \exp\left\{\sum_{j=1}^n G_\sigma(\tau_j - \tau) - G_\sigma(\tau_j)\right\} E_{\mu_{\vt-\tau,\sigma}} \big[e^{n\tau r_{(n)}}h \circ r_{(n)}\big]. 
\end{align*}
Since $h$ is arbitrary, 
\begin{equation}
\label{local density}
  f_{\vt,\sigma}(u) = \exp\left\{n\tau u + \sum_{j=1}^n G_\sigma(\tau_j - \tau) - G_\sigma(\tau_j)\right\}f_{\vt-\tau,\sigma}(u). 
\end{equation}
The relation \eqref{dependence on parameter} then follows from \eqref{micro canonical expectation} and \eqref{local density}. 
In order to define $\vn$ that fulfils \eqref{tension vector function}, observe that as $G_\sigma$ is strictly convex, there is a unique $\tau \in \bR$ such that 
$$
  E_{\mu_{\vt-\tau,\sigma}} [r_{(n)}] = \frac1n\sum_{j=1}^n G'_\sigma(\tau_j - \tau) = u. 
$$
It suffices to define $\vn(u; \vt, \sigma) = \vt - \tau$. 
\end{proof}

Recall the functions $F_\sigma$, $\bst_\sigma$ and $\bar r_\sigma$ defined in \eqref{gibbs potential and free energy}--\eqref{convex conjugate}. 
For each pair of $(\tau ,r) \in \bR^2$, the rate function $I_\sigma(\tau, r)$ is defined as 
\begin{equation}
\label{rate function}
  \begin{aligned}
    I_\sigma(\tau, r) &= G_\sigma(\tau) + F_\sigma(r) - r\tau \\
    &= G_\sigma(\tau) - G_\sigma(\bst_\sigma(r)) - G'_\sigma(\bst_\sigma(r))(\tau - \bst_\sigma(r)). 
  \end{aligned}
\end{equation}
Taking advantage of \eqref{tension vector function} and \eqref{local density}, we can rewrite the density as 
\begin{equation}
\label{local ld}
  f_{\vt,\sigma}(u) = \exp\left\{-\sum_{j=1}^n I_\sigma\big(\tau_j, \bar r_\sigma(\nu_j)\big)\right\}f_{\vn,\sigma}(u), \quad \forall u \in \bR, 
\end{equation}
where $\vn = (\nu_1, \ldots, \nu_n) = \vn(u; \vt, \sigma)$ is defined through \eqref{tension vector function}. 

The classical equivalence of ensembles (cf. \cite[Appendix 2]{KL99}) can be extended to the case that canonical measure is inhomogeneous. 
In order to cover the weakly anharmonic setting in Section \ref{sec:model and results}, for each $n \ge 1$, pick $\sigma_n \in [0, 1)$, $\vt_n = (\tau_{n,1}, \ldots, \tau_{n,n}) \in \bR^n$ and fix them. 
For sake of readability, in the following we write 
$$
  \pi_{n,j} = \pi_{\tau_{n,j},\sigma_n}, \quad \mu_n = \mu_{\vt_n,\sigma_n}, \quad E_n = E_{\mu_n}, \quad \langle \;\cdot\;|u \rangle_n = \langle \;\cdot\;|u \rangle_{\vt_n,\sigma_n}. 
$$
Also denote that 
$$
  u_n = E_n [r_{(n)}], \quad u_{n,2} = \sqrt{E_n \big[(r_{(n)} - u_n)^2\big]}. 
$$
We have the following result (cf. \cite[Corollary A2.1.4, pp. 353]{KL99}). 

\begin{prop}
\label{prop:ee}
Assume some $\epsilon > 0$ and $K > 0$, such that 
\begin{equation}
\label{condition ee 1}
  \sup\{\sigma_n; n \ge 1\} < 1 - \epsilon, \quad \sup\{|\tau_{n,j}|; n \ge 1, 1 \le j \le n\} \le K. 
\end{equation}
For any $F = F(r_1, \ldots, r_k)$ such that $E_n [F^2] < \infty$, we have 
$$
  \big|\langle F|u_n \rangle_n - E_n [F]\big| \le \frac{Ck}n\sqrt{E_n \big[(F - E_n [F])^2\big]}, 
$$
with some constant $C = C_{\epsilon,K}$ for each $n \ge k$. 
\end{prop}

\begin{proof}
In view of \eqref{micro canonical expectation}, the key point is to understand the asymptotic behaviour of the density of $r_{(n)}$. 
To this end, we first check the conditions of the local central limit theorem in Appendix \ref{appendix:local clt}. 
Similarly to Appendix \ref{appendix:tension}, for $0 \le \ell \le 4$, the $\ell$-derivative of $G_\sigma$ satisfies that 
$$
  \big|G_\sigma^{(\ell)}(\tau) - G_0^{(\ell)}(\tau)\big| \le C_1\sigma, \quad \forall \sigma \in [0, 1 - \epsilon),\ \tau \in [-K, K], 
$$
with a uniform constant $C_1 = C_1(\epsilon,K)$. 
Let $\Phi_{\tau,\sigma}$ be the characteristic function 
$$
  \Phi_{\tau,\sigma}(\xi) = \int_\bR \exp\big\{i\xi(r - E_{\pi_{\tau,\sigma}} [r])\big\}\pi_{\tau,\sigma}(dr). 
$$
By the integration by parts formula, 
$$
  i\xi\Phi_{\tau,\sigma}(\xi) = \int_\bR \exp\big\{i\xi(r - E_{\pi_{\tau,\sigma}} [r])\big\}(r + \sigma U'(r) - \tau)\pi_{\tau,\sigma}(dr). 
$$
It is not hard to obtain with some $C_2 = C_2(\epsilon, K)$ that 
$$
  |\Phi_{\tau,\sigma}(\xi)| \le C_2(1 + |\xi|)^{-1}, \quad \forall \sigma \in [0, 1 - \epsilon),\ \tau \in [-K, K]. 
$$
Moreover, using the inequality $|e^x - 1| \le e^{|x|}|x|$, for $0 \le \ell \le 4$, 
%$$
%  \big|\Phi_{\tau,\sigma}(\xi) - \Phi_{\tau,0}(\xi)\big| \le \int_\bR \big|e^{\sigma U} - 1\big|\pi_{\tau,0}(dr) \le C_\epsilon\sigma. 
%$$
$$
  \big|\Phi_{\tau,\sigma}^{(\ell)}(\xi) - \Phi_{\tau,0}^{(\ell)}(\xi)\big| \le C_3\sigma, \quad \forall \sigma \in [0, 1 - \epsilon),\ \tau \in [-K, K], 
$$
with some $C_3 = C_3(\epsilon, K)$. 
By the arguments above, the conditions (\romannum1), (\romannum2), (\romannum3) in Appendix \ref{appendix:local clt} are fulfilled by $\pi_{0,\sigma}$ \emph{uniformly} for $\sigma < 1 - \epsilon$. 
Hence, \eqref{condition ee 1} assures that Lemma \ref{lem:local clt} is applicable to $\mu_n$, even when the reference measure $\pi_{0,\sigma_n}$ is changing with $n$. 

Fix some $k \ge 1$ and a function $F = F(r_1, \ldots, r_k)$. 
Denote by $f_n$ the density of $r_{(n)}$ under $\mu_n$. 
According to Lemma \ref{lem:local clt}, with a bounded sequence $C_{n,0}$, 
$$
  \frac1{\sqrt n}f_n(u_n) = \frac1{u_{n,2}\sqrt{2\pi}}\left(1 + \frac{C_{n,0}}n\right) + o\left(\frac1n\right). 
$$
Similarly, denote by $f_n^*$ the density of $r_{(n-k)}$ under $\mu_{\vt_n^*,\sigma_n}$, $\vt_n^* = (\tau_{n,k+1}, \ldots \tau_{n,n})$, then there are bounded sequences $C_{n,0}^*$, $C_{n,1}^*$, such that 
\begin{align*}
  &\frac1{\sqrt{n - k}}f_n^*\left(\frac{nu_n - kr_{(k)}}{n - k}\right) \\
  =\ &\frac1{u_{n,2}^*\sqrt{2\pi}}\exp\left\{-\frac{y_{(k)}^2}{2(n - k)}\right\}\left(1 + \frac{C_{n,0}^* + C_{n,1}^*y_{(k)}}{n - k}\right) + o\left(\frac1n\right), 
\end{align*}
where 
$$
  u_{n,2}^* = \sqrt{\frac1n\sum_{j=1}^{n-k} G''_{\sigma_n}(\tau_{n,j+k})}, \quad y_{(k)} = \sum_{j=1}^k \frac{r_j - E_n [r_j]}{u_{n,2}^*}. 
$$
Therefore, the density in \eqref{micro canonical expectation} satisfies the estimate 
%$$
%  \begin{aligned}
%    &\frac{n}{n - k}\frac1{f_n(u_n)}f_n^*\left(\frac{nu_n - kr_{(k)}}{n - k}\right) \\
%    =\ &\frac{u_{n,2}\sqrt n}{u_{n,2}^*\sqrt{n-k}}\exp\left\{-\frac{y_{(k)}^2}{2(n - k)}\right\}\left(1 + \frac{C_{n,0}}n\right)^{-1}\left(1 + \frac{C_{n,0}^* + C_{n,1}^*y_{(k)}}n\right) + o\left(\frac1n\right). 
%  \end{aligned}
%$$
\begin{align*}
  &\frac{n}{n - k}\frac1{f_n(u_n)}f_n^*\left(\frac{nu_n - kr_{(k)}}{n - k}\right) \\
  \le\ &\frac{u_{n,2}\sqrt n}{u_{n,2}^*\sqrt{n-k}}\left[1 + \frac Cn\big(1 + y_{(k)} + y_{(k)}^2\big)\right] + o\left(\frac1n\right), 
\end{align*}
where $C = C_{\epsilon,K}$ is a uniform constant. 
Furthermore, 
$$
  \frac{u_{n,2}\sqrt n}{u_{n,2}^*\sqrt{n-k}} = \sqrt{1 + \frac{G''_{\sigma_n}(\tau_{n,1}) + \ldots + G''_{\sigma_n}(\tau_{n,k})}{G''_{\sigma_n}(\tau_{n,k+1}) + \ldots + G''_{\sigma_n}(\tau_{n,n})}} \le 1 + \frac{Ck}{2n} + o\left(\frac1n\right). 
$$
Therefore, with some constant $C' = C'_{\epsilon,K}$, 
$$
  \left|\frac{n}{n - k}\frac1{f_n(u_n)}f_n^*\left(\frac{nu_n - kr_{(k)}}{n - k}\right) - 1\right| \le \frac{C'}n\big(k + y_{(k)} + y_{(k)}^2\big) + o\left(\frac1n\right). 
$$
Proposition \ref{prop:ee} then follows from \eqref{micro canonical expectation} and Schwarz inequality. 
\end{proof}

Proposition \ref{prop:ee} is valid only for cylinder functions $F = F(r_1, \ldots, r_k)$. 
In Section \ref{sec:lemma}, it is required to control the exponential moment of the micro canonical expectation of a particular extensive observation. 
Next, we give the corresponding result. 

Recall that $\bar r_n(\tau) = \bar r_{\sigma_n}(\tau)$, $\tau_n(r) = \bst_{\sigma_n}(r)$. 
Given $F: \bR \to \bR$, let 
$$
  \cF_{n,j}(r) = F(r) - E_{\pi_{n,j}} [F] - \frac d{dr} E_{\pi_{\tau_n(r),\sigma_n}} [F]\Big|_{r=\bar r_n(\tau_{n,j})}\big(r - \bar r_n(\tau_{n,j})\big), 
$$
for $j = 1$, ..., $n$, and define $\cF = \sum_{j=1}^n \cF_{n,j}(r_j)$. 

\begin{prop}
\label{prop:ee exponential}
Assume \eqref{condition ee 1}, and a constant $M$ such that 
\begin{equation}
\label{condition ee 2}
  |\tau_{n,j} - \tau_{n,j+1}| \le Mn^{-\frac32}, \quad \forall n \ge 1,\ 1 \le j \le n. 
\end{equation}
Suppose that for each $n$, $\tau \mapsto \int Fd\pi_{\tau,\sigma_n}$ is twice continuously differentiable, and there is some constant $A > 0$, such that for all $(n, j)$, 
$$
  E_n \big[\exp(s|\cF_{n,j}|)\big] \le e, \quad \forall |s| \le A. 
$$
Then, we can find $A_1 < \infty$ and $A_2 > 0$, such that for all $n \ge 1$, 
$$
  E_n \big[\exp(s|\langle \cF|\:\cdot\; \rangle_n|)\big] \le A_1, \quad \forall |s| \le A_2. 
$$
\end{prop}

\begin{rem}
Proposition \ref{prop:ee exponential} is stated for function $F$ on $\bR$, but the parallel result for $F$ on $\bR^k$ for each $k \ge 1$ can be proved without additional efforts. 
Furthermore, the Euler's constant $e$ in the condition is not sensible. 
\end{rem}

\begin{proof}[Proof of Proposition \ref{prop:ee exponential}]
Fix an $F$ fulfilling the conditions. 
Recall that $u_n = E_n [r_{(n)}]$, and let $A_{n,\delta} = \{\bfr \in \bR^n; r_{(n)} \in (u_n - \delta, u_n + \delta)\}$ for $\delta > 0$. 
Note that 
$$
  E_n \big[e^{s|\langle \cF|u \rangle_n|}\big] = E_n \big[e^{s|\langle \cF|u \rangle_n|}\mathbf1_{A_{n,\delta}^c}\big] + E_n \big[e^{c|\langle \cF|u \rangle_n|}\mathbf1_{A_{n,\delta}}\big]. 
$$
We estimate the two terms respectively. 

For the integral on $A_{n,\delta}^c$, recall the rate function $I_\sigma$ in \eqref{rate function}. 
As $G_{\sigma_n}$ is strictly convex and $\tau_{n,j}$, $\sigma_n$ are bounded, for $\delta$ sufficiently small we have that 
\begin{equation}
\label{rate function estimate}
  I_{\sigma_n}(\tau_{n,j}, r) \ge C\delta^2, \quad \forall |r - \bar r_n(\tau_{n,j})| \ge \delta, 
\end{equation}
with some $C = C(\delta)$. 
By \eqref{local ld}, \eqref{rate function estimate} and Lemma \ref{lem:local clt}, for $\delta$ small but fixed, 
$$
  f_n(u) \le \exp\left\{-M\delta^2n + \frac{\log n}2\right\}, \quad \forall |u - u_n| \ge \delta. 
$$
Hence, by H\"older's inequality, for $p$, $q > 1$ such that $1/p + 1/q = 1$, 
\begin{align*}
  E_n \big[e^{s|\langle \cF|u \rangle_n|}\mathbf1_{A_{n,\delta}^c}\big] &\le \big(\mu_n\{|r_{(n)} - u_n| \ge \delta\}\big)^{\frac1p}\left(E_n \big[e^{sq|\langle \cF|u \rangle_n|}\big]\right)^{\frac1q} \\
  &\le \exp\left\{-\frac{M\delta^2n}p + \frac{\log n}{2p}\right\}\prod_{j=1}^n E_n \big[e^{sq|\cF_{n,j}|}\big]. 
\end{align*}
Choose some $p < M\delta^2 + 1$, we have that for any $|s| < q^{-1}A$ that 
$$
  E_n \big[e^{c|\langle \cF|u \rangle_n|}\mathbf1_{A_{n,\delta}^c}\big] \le \exp\left\{-\frac{M\delta^2n}p + \frac nq + \frac{\log n}{2p}\right\} \to 0. 
$$

To deal with the integral on $A_{n,\delta}$, divide $\langle \cF_{n,j}(r_j)|u \rangle_n$ into two parts: 
\begin{align*}
  K_{n,j} =\ &\langle F(r_j)|u \rangle_n - E_{\pi_{\nu_{n,j},\sigma_n}} [F] \\
  &- \frac d{dr} E_{\pi_{\tau_n(r),\sigma_n}} [F]\Big|_{r=\bar r_n(\tau_{n,j})}\big(\langle r_j|u \rangle_n - \bar r_n(\nu_{n,j})\big); \\
  K'_{n,j} =\ &E_{\pi_{\nu_{n,j},\sigma_n}} [F] - E_{\pi_{\tau_{n,j},\sigma_n}} [F] \\
  &- \frac d{dr} E_{\pi_{\tau_n(r),\sigma_n}} [F]\Big|_{r=\bar r_n(\tau_{n,j})}\big(\bar r_n(\nu_{n,j}) - \bar r_n(\tau_{n,j})\big), 
\end{align*}
where $(\nu_{n,1}, \ldots \nu_{n,n}) = \vn_n = \vn(u; \vt_n, \sigma_n)$ is the vector defined through \eqref{tension vector function}. 
The definition of $\vn_n$ together with Proposition \ref{prop:ee} yields that 
$$
  \langle F(r_j)|u \rangle_n = E_{\nu_{n,j}} [F] + O\left(\frac1n\right), \quad \langle r_j|u \rangle_n = \bar r_n(\nu_{n,j}) + O\left(\frac1n\right), 
$$
uniformly in $A_{n,\delta}$. 
Therefore, $\sum_j |K_{n,j}|$ is uniformly bounded. 
Meanwhile, 
$$
  |K'_{n,j}| \le \frac12\sup_{|u-u_n|<\delta} \left|\frac{d^2}{du^2} E_{\pi_{\tau_n(r),\sigma_n}} [F]\Big|_{r=\bar r_n(\tau_{n,j})}\right|\big(\bar r_n(\nu_{n,j}) - \bar r_n(\tau_{n,j})\big)^2. 
$$
Hence, it suffices to prove that 
$$
  E_n \left[\exp\left\{s\sum_{j=1}^n \big(\bar r_n(\nu_{n,j}) - \bar r_n(\tau_{n,j})\big)^2\right\}\right] \le A_1, \quad \forall |s| \le A_2, 
$$
with some $A_1 < \infty$ and $A_2 > 0$. 
To this end, note that 
$$
  \big(\bar r_n(\nu_{n,j}) - \bar r_n(\tau_{n,j})\big)^2 \le 3(r_{(n)} - u_n)^2 + 3(u_n - \bar r_n(\tau_{n,j}))^2 + 3(\bar r_n(\nu_{n,j}) - r_{(n)})^2. 
$$
We estimate the three terms in the right-hand side respectively. 
For the first term, it is easy to see from central limit theorem that, for $|s| < u_{n,2}/2$, 
$$
  \lim_{n\to\infty} E_n \big[\exp\big\{cn(r_{(n)} - E_n[r_{(n)}])^2\big\}\big] = \frac1{u_{n,2}\sqrt{2\pi}}\int_\bR e^{cx^2}e^{-\frac{x^2}{2u_{n,2}}}dx < \infty
$$
For the second term, taking advantage of \eqref{condition ee 2}, we obtain that 
\begin{align*}
  \sum_{j=1}^n \big(u_n - \bar r_n(\tau_{n,j})\big)^2 &= \sum_{j=1}^n \left(\frac1n\sum_{j'=1}^n \bar r_n(\tau_{n,j'}) - u(\tau_{n,j})\right)^2 \\
  &\le \frac 1n\sum_{j,j'} \big(\bar r_n(\tau_{n,j'}) - \bar r_n(\tau_{n,j})\big)^2 \le O(1). 
\end{align*}
For the third term, observe that by the definition of $\vn_n$, 
$$
  \frac1n\sum_{j=1}^n\bar r_n(\nu_{n,j}) = r_{(n)}, \quad \nu_{n,j'} - \nu_{n,j} = \tau_{n,j'} - \tau_{n,j}, 
$$
so, it can be estimated similarly to the second term. 
\end{proof}

\section{Gradient estimate for the Poisson equation}
\label{sec:poisson}

In this section, we present a gradient-type estimate for the solution to the Poisson equation \eqref{poisson}, which is used in the proof of Lemma \ref{lem:main lemma}. 

We work under the following case with general anharmonic potential function. 
Let $V$ be a given $C^2$-smooth, uniformly convex function: 
$$
  V''(x) \ge c > 0, \quad \forall x \in \bR. 
$$
With a given vector $\mathbf a = (a_1, \ldots, a_n)$, define $U: \bR^n \to \bR$ by 
$$
  U(\bfx) = \sum_{j=1}^n V(x_j) - \mathbf a \cdot \bfx, \quad \forall \bfx \in \bR^n. 
$$
Then, $D_jU = V'(x_{j+1}) - V'(x_j) - a_{j+1} + a_j$, where the operator $D_j$ is 
$$
  D_j = \frac\partial{\partial x_{j+1}} - \frac\partial{\partial x_j}, \quad \forall j = 1, \ldots, n - 1, 
$$
For $x \in \bR$, let $\Sigma_x = \{\bfx \in \bR^n; x_1 + \ldots + x_n = x\}$ be the $(n - 1)$-dimensional hyperplane. 
Suppose a differentiable function $\Psi$ to satisfy the following conditions: 
$$
  \sup_{\bR^n} \sum_{j=1}^{n-1} |D_j\Psi| < \infty, \quad \int_{\Sigma_x} e^{-U(\bfx)}\Psi(\bfx) = 0, 
$$
for all $x \in \Sigma_x$. 
Consider the following partial differential equation: 
$$
  -e^U\sum_{j=1}^{n-1} D_j\big(e^{-U}D_j\psi\big) = \Psi. 
$$
Note that the Poisson equation \eqref{poisson} discussed in the proof of Lemma \ref{lem:main lemma} can be obtained by taking $n = \ell$, $V = \beta V_n$ and $a = \beta(\tau_i, \ldots, \tau_{i+\ell-1})$. 
A sharp gradient-type estimate for the solution $\psi$ is obtained in \cite[Theorem 1.1]{Wu09}. 
By investigating the constant in their estimate, we get the following result. 

\begin{prop}
\label{prop:lipschitz estimate}
There is a constant $C$ dependent on $c = \inf V''$, such that 
$$
  \big|\mathbf D\psi(\bfx)\big|^2 \le Cn^4\sup_{\bR^n} \big|\mathbf D\Psi\big|^2, \quad \forall \bfx \in \bR^n, 
$$
where $\mathbf D = (D_1, \ldots, D_{n-1})$. 
\end{prop}

\begin{proof}
Rewrite the equation with the new coordinates: 
$$
  y_j = -\sum_{i=1}^j x_i, \quad \forall j = 1, \ldots, n - 1, \quad y_* = -\sum_{j=1}^n x_j. 
$$
Notice that for $1 \le j \le n - 1$, $D_j = \partial_{y_j}$. 
The new equation is 
$$
  \nabla_\bfy \widetilde U(\bfy; y_*) \cdot \nabla_\bfy \widetilde\psi - \Delta_\bfy \widetilde\psi = \widetilde\Psi(\bfy; y_*), 
$$
where $y_*$ is viewed as a parameter, $\bfy = (y_1, \ldots, y_{n-1})$, and 
\begin{align*}
    &\widetilde\psi(\bfy; y_*) = \psi(\bfx), \quad \widetilde\Psi(\bfy; y_*) = \Psi(\bfx), \\
    &\widetilde U = V(-y_1) + \sum_{j=1}^{n-2} V(y_j - y_{j+1}) + V(y_{n-1} - y_*) + \sum_{j=1}^{n-1} (a_j - a_{j+1})y_j. 
\end{align*}
Denote by $\lambda_n = \lambda_{\mathrm{min}}(H_n)$ the smallest eigenvalue of 
\begin{equation}
\label{hessian}
  H_n = \hess_\bfy \widetilde U(\cdot, y_*) = 
  \begin{pmatrix}
    b_1 + b_2 &-b_2 &0 &\dots &0 \\
    -b_2 & b_2 + b_3 &-b_3 &\dots &0 \\
    0 &-b_3 &b_3 + b_4 &\dots &0 \\
    \vdots &\vdots &\vdots & &\vdots \\
    0 &0 &0 &\dots &b_{n-1} + b_n
  \end{pmatrix}, 
\end{equation}
where we write $b_j = V''(x_j)$ for $1 \le j \le n$. 
As each $b_j > 0$, it is easy to observe that $\lambda_n > 0$. 
Applying \cite[Theorem 1.1]{Wu09} for each fixed $y_*$, 
$$
  \left|\nabla_\bfy\widetilde\psi(\bfy; y_*)\right| \le \lambda_n^{-1}\sup_\bfy \left|\nabla_\bfy\widetilde\Psi(\bfy; y_*)\right|, \quad \forall (\bfy, y_*) \in \bR^n. 
$$
In Lemma \ref{lem:eigenvalue}, we show that $\lambda_n \ge Cn^{-2}$ with some constant $C = C(c)$. 
By returning to the original variables $\bfx$, we get the desired estimate. 
\end{proof}

The proof of Proposition \ref{prop:lipschitz estimate} is completed by the following lower bound of $\lambda_n$. 

\begin{lem}
\label{lem:eigenvalue}
In \eqref{hessian}, suppose that $b_j \ge c > 0$ for all $j$, then 
\begin{equation}
  \lambda_n = \lambda_{\mathrm{min}}(H_n) \ge \frac{6c}{(n-1)(n+1)}, \quad \forall n \ge 2. 
\end{equation}
\end{lem}

\begin{proof}
Let $I_n$ be the $n\times n$ identity matrix, and define $Q_1(\lambda) = -b_1^{-1}$, 
\begin{equation}
\label{eigenvalue 1}
  Q_n(\lambda) = (-1)^n\det\big(\lambda I_{n-1} - H_n\big)\prod_{j=1}^n \frac1{b_n}, \quad \forall n \ge 2. 
\end{equation}
%First notice that 
%\begin{equation}
%  P_n = (\lambda - b_{n-1} - b_n)P_{n-1} - b_{n-1}^2P_{n-2}, \quad \forall n \ge 3, 
%\end{equation}
%and $P_2 = \lambda - b_1 - b_2$. Now define that 
%$$
%  Q_n(\lambda) = (-1)^nP_n(\lambda)\prod_{j=1}^n \frac1{b_n}. 
%$$
Notice that $Q_2(0) = -(b_1^{-1} + b_2^{-1})$, and 
$$
  \frac{Q_n(0) - Q_{n-1}(0)}{Q_{n-1}(0) - Q_{n-2}(0)} = \frac{b_{n-1}}{b_n}, \quad \forall n \ge 3. 
$$
By a simple inductive argument, we obtain that 
$$
  Q_n(0) = -\sum_{j=1}^n \frac1{b_j}, \quad \forall n \ge 1. 
$$
Similarly, we have $Q'_1(0) = 0$, $Q'_2(0) = (b_1b_2)^{-1}$, and 
%$$
%  P'_n(0) = P_{n-1}(0) - (b_{n-1} + b_n)P'_{n-1}(0) - b_{n-1}^2P'_{n-2}(0)
%$$
%$$
%  Q'_n(0) = -\frac{Q_{n-1}(0)}{b_n} + \left(\frac{b_{n-1}}{b_n} + 1\right)Q'_{n-1}(0) - \frac{b_{n-1}}{b_n}Q'_{n-2}(0)
%$$
$$
  b_n\big(Q'_n(0) - Q'_{n-1}(0)\big) - b_{n-1}\big(Q'_{n-1}(0) - Q'_{n-2}(0)\big) = -Q_{n-1}(0) > 0, \quad \forall n \ge 3. 
$$
By using this relation recurrently, we have the expression 
$$
  Q'_n(0) = -\sum_{j'=2}^n \frac1{b_{j'}}\sum_{j=1}^{j'-1} Q_j(0) = \sum_{j'=1}^{n-1} \left(\sum_{j=1}^{j'} \frac1{b_j}\right)\left(\sum_{j=j'+1}^n \frac1{b_j}\right). 
$$
Observing that for each $1 \le j' \le n - 1$, 
$$
  \left(\sum_{j=1}^{j'} \frac1{b_j} - \frac{j'}n\sum_{j=1}^n \frac1{b_j}\right)\left(\sum_{j=j'+1}^n \frac1{b_j} - \frac{n-j'}n\sum_{j=1}^n \frac1{b_j}\right) \le 0. 
$$
Therefore, with the condition $b_j \ge c > 0$ for each $j$, we get 
\begin{align*}
  \left(\sum_{j=1}^{j'} \frac1{b_j}\right)\left(\sum_{j=j'+1}^n \frac1{b_j}\right) &\le \left[\frac{(j')^2}{n^2}\sum_{j=j'+1}^n \frac1{b_j} + \frac{(n - j')^2}{n^2}\sum_{j=1}^{j'} \frac1{b_j}\right]\sum_{j=1}^n \frac1{b_j} \\
  &\le \frac{(j')^2(n - j') + (n - j')^2j'}{cn^2}\sum_{j=1}^n \frac1{b_j} \\
  &= \frac{j'(n - j')}{cn}\sum_{j=1}^n \frac1{b_j}. 
\end{align*}
Summing up the estimate above for $j' = 1$ to $n - 1$, 
\begin{equation}
\label{eigenvalue 2}
  0 < Q'_n(0) \le \sum_{j'=1}^{n-1}\frac{j'(n - j')}{nc}\sum_{j=1}^n \frac1{b_j} = \frac{(n - 1)(n + 1)}{6c}\sum_{j=1}^n \frac1{b_j}. 
\end{equation}
Note that all the roots of $Q_n$ are real and positive, so $\lambda_{\mathrm{min}}$ is the first root to the right of the origin. 
With this observation, \eqref{eigenvalue 1} and \eqref{eigenvalue 2} assure that 
$$
  \lambda_{\mathrm{min}}(H_n) \ge -\frac{Q_n(0)}{Q'_n(0)} \ge \frac{6c}{(n - 1)(n + 1)}. 
$$
The lower bound for $\lambda_{\mathrm{min}}$ then follows. 
\end{proof}

\section*{Acknowledgments} 
The author thanks Stefano Olla for the very helpful discussions and insightful advices. 
This work has been supported by the grants ANR-15-CE40-0020-01 LSD 
of the French National Research Agency. 

%%%%%%%%%%%%%%%%%%%%%%%%%%%%%%%%%%%%%%%%%
%appendix

\begin{appendices}

%\appendixtotoc
%\renewcommand{\appendixname}{Appendix}

\titleformat{\section}[hang]
{\bfseries\large}{\appendixname\ \thesection.}{0.5em}{}[]
\titlespacing*{\section}{0em}{2em}{1.5em}

%\titleformat{\subsection}[runin]
%{\bfseries\normalsize}{\thesubsection.}{0em}{\ }[.]

\section{Equilibrium tension}
\label{appendix:tension}

Recall the probability measure $\pi_{\tau,\sigma}$ defined in \eqref{one site distribution}, and the normalization constant $Z_\sigma(\tau)$ appeared in it. 
Note that for $\beta > 0$ and $\sigma = 0$, 
$$
  Z_0(\tau) = \sqrt{\frac{2\pi}\beta}\exp\left\{\frac{\beta\tau^2}2\right\}, \quad \forall \tau \in \bR. 
$$
Denote by $\langle\;\cdot\;\rangle_{\tau,\sigma}$ the integral with respect to $\pi_{\tau,\sigma}$. 
For any $\epsilon > 0$ and $\sigma \in [0, 1 - \epsilon]$, with the elementary inequality $|e^x - 1 - x| \le e^{|x|}x^2/2$ we can get that 
%$$
%  \begin{aligned}
%    &\big|Z_\sigma(\tau) - Z_0(\tau)\big(1 - \sigma\beta\langle U \rangle_{\tau,0}\big)\big| \\
%    \le\ &\int_\bR \exp\left\{-\frac{\beta r^2}2 + \beta\tau r\right\}\big|e^{-\sigma\beta U(r)} - 1 + \sigma\beta U(r)\big|dr \\
%    \le\ &\frac{\sigma^2\beta^2}2\int_\bR \exp\left\{-\frac{\beta r^2}2 + \beta\tau r + \sigma\beta|U(r)|\right\}U^2(r)dr \\
%    \le &\frac{\sigma^2\beta^2}2\int_\bR \exp\left\{-\frac{\beta r^2}2 + \beta\tau r + (1 - \epsilon)\beta|U(r)|\right\}U^2(r)dr
%  \end{aligned}
%$$
\begin{align*}
  &\big|Z_\sigma(\tau) - Z_0(\tau)\big(1 - \sigma\beta\langle U \rangle_{\tau,0}\big)\big| \le C\sigma^2, \\
  &\big|\beta^{-1}Z'_\sigma(\tau) - Z_0(\tau)\big(\tau - \sigma\beta\langle rU \rangle_{\tau,0}\big)\big| \le C\sigma^2, \\
  &\big|\beta^{-1}Z''_\sigma(\tau) - Z_0(\tau)\big(\beta\tau^2 + 1 - \sigma\beta^2\langle r^2U \rangle_{\tau,0}\big)\big| \le C\sigma^2, \\
  &\big|\beta^{-2}Z'''_\sigma(\tau) - Z_0(\tau)\big(\beta\tau^3 + 3\tau -\sigma\beta^2\langle r^3U \rangle_{\tau,0}\big)\big| \le C\sigma^2, 
\end{align*}
with some constant $C = C_{\beta,\tau,\epsilon}$. 
Furthermore, the constant $C$ can be taken uniformly for $\tau$ in any compact intervals in $\bR$. 

Recall the functions $\bar r_\sigma$ and $\bst_\sigma$ defined through \eqref{gibbs potential and free energy}--\eqref{convex conjugate}. 
From the definition and the estimate above, we obtain that as $\sigma \to 0^+$, 
\begin{align*}
  &\bar r_\sigma(\tau) = \tau - \sigma\beta\big\langle (r - \tau)U \big\rangle_{\tau,0} + o_{\beta,\tau}(\sigma), \\
  &\bar r'_{\sigma}(\tau) = 1 - \sigma\beta^2\big\langle [(r - \tau)^2 - \beta^{-1}]U \big\rangle_{\tau,0} + o_{\beta,\tau}(\sigma), \\
  &\bar r''_{\sigma}(\tau) = -\sigma\beta^3\big\langle [(r - \tau)^3 - 3\beta^{-1}(r - \tau)]U \big\rangle_{\tau,0} + o_{\beta,\tau}(\sigma), 
\end{align*}
uniformly for $\tau$ in any compact interval. 
As the \emph{macroscopic tension function} $\bst_\sigma$ is the inverse of $\bar r_\sigma$, we can conclude the following asymptotic behaviours 
\begin{align*}
  &\bst_\sigma(r) = r + C_0(\beta, r)\sigma + o_{\beta,r}(\sigma), \\
  &\bst'_\sigma(r) = 1 + C_1(\beta, r)\sigma + o_{\beta,r}(\sigma), \\
  &\bst''_\sigma(r) = C_2(\beta, r)\sigma + o_{\beta,r}(\sigma), 
\end{align*}
holds uniformly for $r$ in any compact intervals in $\bR$. 
Moreover, the constants $C_0$, $C_1$ and $C_2$ are continuously dependent on $\beta$ and $r$. 

%The central argument is to prove that for any given closed interval $A = [a_1, a_2]$, there exists a closed interval $B_A$, a positive number $\delta_A$, such that 
%$$
%  \bar r_\sigma(B_A) \supseteq A, \quad \forall \sigma \in [0, \delta_A]. 
%$$
%By this we know that for all $r \in A$, $\tau_\sigma = \bst_\sigma(r) \in B_A$, so that 
%$$
%  \sup_{r\in A} \big|\bst_\sigma(r) - r\big| \le \sup_{\tau\in B_A} \big|\tau - \bar r_\sigma(\tau)\big| \le C\sigma. 
%$$
%Indeed, we can pick $B_A = [a_1 - 1, a_2 + 1]$. 
%There exist some numbers $C_A$ and $\delta'_A$, such that $|\tau - \bar r_\sigma(\tau)| \le C_A\sigma$ for $\tau \in B_A$ and $|\sigma| < \delta'_A$. 
%With $\delta_A = \delta'_A \land C_A^{-1}$, 
%$$
%  a_1 \ge a_1 - 1 + C_A\sigma \ge \bar r_\sigma(a_1 - 1), \quad \forall \sigma \in [0, \delta_A]. 
%$$
%Similarly, we have $a_2 \le \bar r_\sigma(a_2 + 1)$. 
%As $\bar r_\sigma$ is continuous, $\bar r_\sigma(B_A) \supseteq A$. 

\section{Quasi-linear $p$-system}
\label{appendix:p-system}

In this appendix we present a lower bound for the life span of the classical solution of a quasi-linear $p$-system with smooth initial data. 
The result is necessary for the proof of Proposition \ref{prop:smooth regime}. 

Suppose that $f$ is a positive function in $C^1(\bR)$. 
Consider the following partial differential equations for $t \ge 0$ and $x \in \bT$: 
\begin{equation}
\label{p-system}
  \partial_tp(t, x) = f(r)\partial_xr(t, x), \quad \partial_tr(t, x) = \partial_xp(t, x), 
\end{equation}
with some given smooth initial data 
$$
  p(0, \cdot) = p_0 \in C^1(\bT), \quad r(0, \cdot) = r_0 \in C^1(\bT). 
$$
Note that by taking $f = \bst'_\sigma$, \eqref{p-system} coincides the hydrodynamic equation \eqref{quasi linear p} for anharmonic potential. 
It is well-known that if $f \not= const$, \eqref{p-system} would produce shocks in finite time. 
Recall that $|\cdot|_\bT$ represents the uniform norm on $\bT$, and define 
$$
  K = |p_0|_\bT + |r_0|_\bT\sup \left\{\sqrt{f(r)};\ |r| \le |r_0|_\bT\right\}. 
$$
The next lemma is a special case of the classical result in \cite{Lax64}. 

\begin{lem}
\label{lem:generation of shock}
Smooth solution of \eqref{p-system} exists on $t \in [0, T]$ for any 
$$
  T < T_* = 4\left|p'_0\sqrt{f(r_0)} + r'_0f(r_0)\right|_\bT^{-1}\left(\sup_{|r|\le K} \left|f^{-\frac54}(r)f'(r)\right|\right)^{-1}. 
$$
\end{lem}

\begin{rem}
For the readers not familiar to the hyperbolic systems, it is worth mentioning that the bound we obtained above is not as sharp as the case of scalar equation, for instance the inviscid Burger's equation. 
\end{rem}

\begin{proof}
We briefly state the proof. 
Define an antiderivative of $\sqrt f$: 
$$
  F(s) = \int_0^s \sqrt{f(r)}dr, \quad \forall s \in \bR. 
$$
The equation can be rewritten in Riemann invariants as 
$$
  \partial_tu = \lambda(u, v)\partial_xu, \quad \partial_tv = -\lambda(u, v)\partial_xv, \quad (u, v)(0, \cdot) = (u_0, v_0), 
$$
where $u = p + F(r)$, $v = p - F(r)$ and $\lambda(u, v) = \sqrt{f(r)}$. 

Consider the characteristic lines $(t, x_{\pm,t})$, given by the ODEs 
$$
  \frac{dx_t}{dt} = \pm\lambda\big(u(t, x_t), v(t, x_t)\big), \quad x_0 = x \in \bT. 
$$
Within the life span of the smooth solution, $u$ is constant along $(t, x_{+,t})$, thus 
\begin{equation}
\label{priori bound}
   \sup_{x\in\bT} |u(t, x)| \le \sup_{x\in\bT} |u_0(x)| \le K. 
\end{equation}
Similarly, we have a priori bound for $v(t, x)$. 

Suppose that the smooth solution of \eqref{p-system} exists on time interval $[0, T]$ for some $T > 0$. 
Taking spatial derivative on the equation of $u$, 
$$
  \partial_{tx}u - \lambda\partial_{xx}u = \partial_u\lambda(\partial_xu)^2 + \partial_v\lambda\partial_xu\partial_xv, \quad t \in [0, T]. 
$$
In order to investigating the continuity, let $z(t, x) = \sqrt{\lambda(u, v)}\partial_xu$. 
From the equation above, for $t \in [0, T]$, $z$ solves the Riemann problem given by 
$$
  \left\{
  \begin{aligned}
    &\partial_tz - \lambda\partial_xz = \Lambda z^2, \quad \Lambda = 2\partial_u\sqrt\lambda, \\
    &z(0, \cdot) = \sqrt{\lambda(u_0, v_0)}u'_0. 
  \end{aligned}
  \right.
$$
%Indeed, 
%\begin{align*}
%  \partial_tz - \lambda\partial_xz &= \frac{\partial_xu}{2\sqrt\lambda}\big(\partial_t\lambda - \lambda\partial_x\lambda\big) + \sqrt\lambda\partial_{tx}u - \lambda^{\frac32}\partial_{xx}u \\
%  &= \frac{\partial_v\lambda\partial_xu}{2\sqrt\lambda}\big(\partial_tv - \lambda\partial_xv\big) + \sqrt\lambda\partial_u\lambda(\partial_xu)^2 + \sqrt\lambda\partial_v\lambda\partial_xu\partial_xv \\
%  &= \frac{\partial_v\lambda\partial_xu}{2\sqrt\lambda}\big(\partial_tv + \lambda\partial_xv\big) + \sqrt\lambda\partial_u\lambda(\partial_xu)^2 = \sqrt\lambda\partial_u\lambda(\partial_xu)^2. 
%\end{align*}
By \eqref{priori bound}, before the generation of shocks, $|\Lambda|$ is bounded from above by 
$$
  K' \triangleq 2\sup \left\{\partial_u\sqrt{\lambda(u, v)};\ |u|, |v| \le K\right\}. 
$$
Via a comparison argument, one obtains that $|z(t, x)| < \infty$ for 
$$
  (t, x) \in \left[0,\ \frac1{K'\sup_{x\in\bT} |z(0,x)|}\right) \times \bT, 
$$
which guarantees that $|\partial_xu| < \infty$, so shock cannot form. 
Since 
$$
  2\partial_u\sqrt\lambda = \frac{\partial_r\sqrt\lambda}{F'(r)} = 4^{-1}f^{-\frac54}(r)f'(r), 
$$
the estimate in Lemma \ref{lem:generation of shock} then follows. 
\end{proof}

\section{Yau's entropy method}
\label{appendix:yau}

In this appendix, we apply Yau's relative entropy method to obtain the formulas \eqref{yau 1}--\eqref{yau 3} in the proof of Theorem \ref{thm:relative entropy}. 

Fix $n \ge 1$ and $\sigma \in (0, 1)$. 
Take a smooth function $(p, r) = (p, r)(t ,x)$ on $[0, T] \times \bT$, and define $p_i^n = p(t, i/n)$, $r_i^n = r(t, i/n)$, $\tau_i^n = \bst_\sigma(r_i^n)$ for each $i \in \bT_n$, where $\bst_\sigma$ is given by \eqref{convex conjugate}. 
Recall the Gibbs states defined in \eqref{gibbs states} and choose $\nu = \nu_{0,0,\sigma}^n$ as the reference measure on $\Omega_n$. 
Consider the local Gibbs measure $d\mu_t = \exp(\beta\varphi_t)d\nu$, where 
$$
  \varphi_t(\ve) = \sum_{i\in\bT_n} \big(p_i^np_i + \tau_i^nr_i\big) + \sum_{i\in\bT_n} \left[-\frac{(p_i^n)^2}2 + G_\sigma(0) - G_\sigma(\tau_i^n)\right]. 
$$
Let $\ve(t)$ be the Markov process generated by $\cL_{n,\sigma,\gamma}$ in \eqref{generator} with some fixed $\gamma > 0$, and denote by $f_t$ the density of $\ve(t)$ with respect to $\mu_t$. 

From the definition of the relative entropy in \eqref{entropy}, 
%\begin{equation}
%  \begin{aligned}
%    \frac d{dt}H(f_t; \mu_t) &= \int \frac d{dt}\big(e^{\beta\varphi_t}f_t\big) \cdot \log f_td\nu + \int e^{\beta\varphi_t}f_t \cdot f_t^{-1}\frac d{dt}f_td\nu \\
%    &= \int \frac d{dt}\big(e^{\beta\varphi_t}f_t\big) \cdot (\log f_t + 1)d\nu - \int f_t\frac d{dt}e^{\beta\varphi_t}d\nu. 
%  \end{aligned}
%\end{equation}
%Since $\nu$ is stationary for $\ve(\cdot)$, the Fokker--Planck equation reads 
%\begin{equation}
%  \frac d{dt}(e^{\beta\varphi_t}f_t) = \cL_{n,\sigma,\gamma}^*(e^{\beta\varphi_t}f_t), 
%\end{equation}
%where $\cL_{n,\sigma,\gamma}^*$ is the adjoint of $\cL_{n,\sigma,\gamma}$ with respect to $\nu$. 
%Therefore, 
%$$
%  \frac d{dt}H(f_t; \mu_t) = \int f_t\left[\cL_{n,\sigma,\gamma}[\log f_t] - \frac d{dt}(\beta\varphi_t)\right]d\mu_t. 
%$$
%Note that 
%\begin{align*}
%  \int f_t\cA_{n,\sigma}[\log f_t]d\mu_t &= \int \cA_{n,\sigma}f_t d\mu_t, \\
%  \int f_t\cS_{n,\sigma}[\log f_t]d\mu_t &= -\frac12\sum_{i\in\bT_n} \int \cY_i[\log f_t]\cY_i\big[e^{\beta\varphi_t}f_t\big]d\nu \\
%  &= -\frac12\sum_{i\in\bT_n} \int \frac1{f_t}(\cY_if_t)^2d\mu_t - \frac12\sum_{i\in\bT_n} \int \cY_if_t \cdot \cY_i\big[e^{\beta\varphi_t}\big]d\nu \\
%  &= -\frac12\sum_{i\in\bT_n} \int \frac1{f_t}(\cY_if_t)^2d\mu_t + \int \cS_{n,\sigma}f_td\mu_t, 
%\end{align*}
%and similar computation holds for $\overline\cS_n$, so that 
$$
  \frac d{dt}H(f_t; \mu_t) = -4n\gamma D\left(\sqrt{f_t}, \mu_t\right) + \int \left(\cL_{n,\sigma,\gamma}f_t - \beta f_t\frac d{dt}\varphi_t\right)d\mu_t, 
$$
where the Dirichlet form $D(f, \mu)$ is defined as 
%\begin{equation}
%\label{dirichlet}
$$
  D(f, \mu) = \int \Gamma_nfd\mu, \quad %\Gamma_nf = \frac12\sum_{i\in\bT_n} (\cY_if)^2 + (\cX_if)^2, 
  \Gamma_nf = \frac12\sum_{i\in\bT_n} (\cY_if)^2, 
%\end{equation}
$$
for probability measure $\mu$ and density function $f$ on $\Omega_n$. 
Since 
\begin{align*}
  \int \cA_{n,\sigma}f_td\mu_t &= -\int f_t\cA_{n,\sigma}\big[e^{\beta\varphi_t}\big]d\nu = -\beta\int f_t\cA_{n,\sigma}\varphi_td\mu_t, \\
  \int \cS_{n,\sigma}f_td\mu_t &= -\frac12\sum_{i\in\bT_n} \int \cY_if_t \cdot \cY_i\big[e^{\beta\varphi_t}\big]d\nu \\
  &\le \frac14\sum_{i\in\bT_n} \int \frac1{f_t}(\cY_if_t)^2d\mu_t + \frac{\beta^2}4\sum_{i\in\bT_n} \int f_t(\cY_i\varphi_t)^2d\mu_t, 
\end{align*}
%and the similar inequality for $\cX_i$, we obtain that 
we obtain that with $J_t^n = -(n\cA_{n,\sigma} + d/dt)\varphi_t$, 
\begin{equation}
\label{general yau 1}
%  \frac d{dt}H(f_t; \mu_t) \le - 2n\gamma D\left(\sqrt{f_t}; \mu_t\right) + \beta\int f_tJ_t^nd\mu_t + \frac{\beta^2n\gamma}2\int f_t(\Gamma_n\varphi_t)d\mu_t, 
  \frac d{dt}H(f_t; \mu_t) \le - 2n\gamma D\left(\sqrt{f_t}; \mu_t\right) + \beta\int f_tJ_t^nd\mu_t + \frac{\beta^2n\gamma}2\int f_t(\Gamma_n\varphi_t)d\mu_t. 
\end{equation}
%where $J_t^n = -(n\cA_{n,\sigma} + d/dt)\varphi_t$. 
Using the explicit formula of $\varphi_t$, 
$$
%  \Gamma_n\varphi_t = \frac12\sum_{i\in\bT_n} \big(p_{i+1}^n(t) - p_i^n(t)\big)^2 + \big(\tau_{i+1}^n(t) - \tau_i^n(t)\big)^2 \le \frac{C_T}n. 
  \Gamma_n\varphi_t = \frac12\sum_{i\in\bT_n} \big(\tau_{i+1}^n(t) - \tau_i^n(t)\big)^2 \le \frac1n\int_\bT |\partial_x\tau(r(t, x))|^2dx, 
$$
%with a constant $C_T$ uniform for small $\sigma$. 
so that $\Gamma_n\varphi_t \le C_T/n$. 
Also, by the formula of $\varphi_t$, 
\begin{align*}
  \cA_{n,\sigma}\varphi_t &= \sum \tau_i^n(p_i - p_{i-1}) + p_i^n\big(V'_\sigma(r_{i+1}) - V'_\sigma(r_i)\big) \\
  &= -\sum_{i\in\bT_n} \begin{pmatrix}\tau_{i+1}^n - \tau_i^n \\ p_i^n - p_{i-1}^n\end{pmatrix} \cdot \begin{pmatrix}p_i - p_i^n \\ V'_\sigma(r_i) - \tau_i^n\end{pmatrix}, \\
  \frac d{dt}\varphi_t &= \sum_{i\in\bT_n} \frac{dp_i^n}{dt}(p_i - p_i^n) + \sum_{i\in\bT_n} \frac{d\tau_i^n}{dt}(r_i - r_i^n) \\
    &= \sum_{i\in\bT_n} \frac d{dt}\begin{pmatrix}p_i^n \\ r_i^n\end{pmatrix} \cdot \begin{pmatrix}p_i - p_i^n \\ \bst'_\sigma(r_i^n)(r_i - r_i^n)\end{pmatrix}. 
\end{align*}
Therefore, we obtain the explicit form of $J_t^n$ as 
\begin{equation}
\label{general yau 2}
  \begin{aligned}
    J_t^n = &\sum_{i\in\bT_n} \left[-\frac d{dt}\begin{pmatrix}p_i^n \\ r_i^n\end{pmatrix} + n\begin{pmatrix}\tau_{i+1}^n - \tau_i^n \\ p_i^n - p_{i-1}^n\end{pmatrix}\right] \cdot \begin{pmatrix}p_i - p_i^n \\ V'_\sigma(r_i) - \tau_i^n\end{pmatrix} \\
    &+ \sum_{i\in\bT_n} \frac{dr_i^n}{dt} \cdot \big[V'_\sigma(r_i) - \tau_i^n - \bst'_\sigma(r_i^n)(r_i - r_i^n)\big]. 
  \end{aligned}
\end{equation}
In particular, the formulas \eqref{yau 1}--\eqref{yau 3} follow from \eqref{general yau 1}, \eqref{general yau 2} by taking $\sigma = \sigma_n$, $\gamma = \gamma_n$ and $(p, r)$ to be the solution $(\fp_n, \fr_n)$ of the hydrodynamic equation \eqref{quasi linear p} for $\sigma = \sigma_n$. 

\section{Entropy and moment inequalities}
\label{appendix:inequalities}

Recall the relative entropy $H(f; \mu)$ in \eqref{entropy} for probability measure $\mu$ and density function $f$ on some measurable space $\Omega$. 
In this appendix we give some classical inequalities related to $H(f; \mu)$. 
We begin from a variational formula of $H(f; \mu)$: 
\begin{equation}
\label{entropy variational}
  H(f; \mu) = \sup_{g\in B_b(\Omega)} \left\{\int_\Omega fgd\mu - \log\int_\Omega e^gd\mu\right\}, 
\end{equation}
where $B_b(\Omega)$ stands for the class of all bounded measurable functions on $\Omega$. 
The proof of \eqref{entropy variational} can be found in \cite[Theorem 4.1]{Varadh88}. 
From \eqref{entropy variational} we immediately get the first lemma. 

\begin{lem}
Let $(\Omega_1, \sF_1, \mu_1)$, $(\Omega_2, \sF_2, \mu_2)$ be two probability spaces. Suppose $f$ to be a density function on $\Omega = \Omega_1 \times \Omega_2$ with respect to $\mu = \mu_1 \otimes \mu_2$, then 
\begin{equation}
\label{entropy inequality 0}
  H(f_1; \mu_1) \le H(f; \mu), 
\end{equation}
where $f_1$ is the density of the marginal distribution of $fd\mu$ on $\Omega_1$. 
\end{lem}

%\begin{proof}
%$$
%  f_1(x) = \int_{\Omega_2} f(x, y)\mu_2(dy), \quad \forall x \in \Omega_1. 
%$$
%\begin{align*}
%  H(f_1; \mu_1) &= \sup_{g_*\in B_b(\Omega_1)} \left\{\int_{\Omega_1} f_1g_*d\mu_1 - \log\int_{\Omega_1} e^{g_*}d\mu_1\right\} \\
%  &= \sup_{g_*\in B_b(\Omega_1)} \left\{\int_\Omega fg_*d\mu - \log\int_\Omega e^{g_*}d\mu\right\} \le H(f; \mu). 
%\end{align*}
%\end{proof}

Next we give two inequalities frequently used in this article. 

\begin{lem}
For any measurable subset $A \subseteq \Omega$, 
\begin{equation}
\label{entropy inequality 1}
  \int_A fd\mu \le \frac{H(f; \mu) + \log2}{-\log\mu(A)}. 
\end{equation}
If $X: \Omega \to \bR$ is integrable under $fd\mu$, then for any $\alpha > 0$, 
\begin{equation}
\label{entropy inequality 2}
  \int_\Omega fXd\mu \le \frac1\alpha\left[H(f; \mu) + \log\int_\Omega e^{\alpha X}d\mu\right]. 
\end{equation}
\end{lem}

\begin{proof}
Taking $g = -\log(\mu(A))\mathbf1_A$ in \eqref{entropy variational}, we obtain that 
$$
  H(f; \mu) \ge -\log\big(\mu(A)\big)\int_A fd\mu - \log\big(2 - \mu(A)\big), 
$$
and \eqref{entropy inequality 1} follows. 
For \eqref{entropy inequality 2}, if $X$ is bounded, take $g = \alpha X$ to get 
$$
  H(f; \mu) \ge \alpha\int_\Omega fXd\mu - \log\int_\Omega e^{\alpha X}d\mu. 
$$
We can obtain \eqref{entropy inequality 2} via a standard approximating argument. 
\end{proof}

A family of random variables $\{X_i; i = 1, \ldots, m\}$ is said to be $\ell$-independent for some $1 \le \ell \le m$, if for any subset $\Gamma \subseteq \{1, \ldots, m\}$ such that $|i - j| \ge \ell$ for each $i \not= j \in \Gamma$, the sub family $\{X_i; i \in \Gamma\}$ is independent. 
From \eqref{entropy inequality 2} we easily get the next lemma. 

\begin{lem}
If $\{X_1, X_2, \ldots, X_m\}$ is $\ell$-independent, then for any $\alpha > 0$, 
$$
  \left|\int f\sum_{i=1}^m X_id\mu\right| \le\ \frac1\alpha\left[H(f; \mu) + \frac1\ell\sum_{i=1}^m \max\left\{\log\int e^{\pm\alpha\ell X_i}d\mu\right\}\right]. 
$$
\end{lem}

\begin{proof}
For $k = 0$, $1$, ..., $\ell - 1$, let $\Gamma_k = \{k + i\ell; 1 \le i \le (m - k)/\ell\}$. 
Since $\{X_i, i \in \Gamma_k\}$ is independent, \eqref{entropy inequality 2} yields that 
$$
  \int f\sum_{i\in\Gamma_k} X_id\mu \le \frac1{\alpha\ell}\left[H(f; \mu) + \sum_{i\in\Gamma_k} \log\int e^{\alpha\ell X_i}d\mu\right]. 
$$
Taking summation over $k \in \{0, 1, \ldots, \ell - 1\}$, we get that 
\begin{equation}
\label{entropy inequality 3}
  \int f\sum_{i=1}^m X_id\mu \le \frac1\alpha\left[H(f; \mu) + \frac1\ell\sum_{i=1}^m \log\int e^{\alpha\ell X_i}d\mu\right]. 
\end{equation}
The proof is completed by repeating the argument with $-X_i$ instead of $X_i$. 
\end{proof}

Taking $A = \{|X|>\lambda\}$ in \eqref{entropy inequality 1} gives us tail estimates of $X$. 
The following result makes it possible to get moment bounds of $X$ from tail estimates. 
It has been used in the proof of Corollary \ref{cor:general quantitative hl} and the tightness of the fluctuation field. 

\begin{lem}
\label{lem:moment}
Suppose a constant $C > 0$, some $q > 1$ such that 
$$
  P(|X| > \lambda) \le C\lambda^{-q}, \quad \forall \lambda > 0. 
$$
Then, for any $q \in [1, p)$, there exists a constant $K_{p,q} > 1$, such that 
$$
  E \big[|X|^p\big] \le K_{p,q}C^{\frac pq}. 
$$
\end{lem}

\begin{proof}
Using the integration-by-parts formula, for all $1 \le p < q$, 
\begin{align*}
  E \big[|X|^p\big] &\le \int_0^\infty \frac d{d\lambda}\big(\lambda^p\big)P(|X| \ge \lambda)d\lambda \\
  &\le \int_{0\le\lambda<C^{\frac1q}} p\lambda^{p-1}d\lambda + C\int_{\lambda\ge C^{\frac1q}} p\lambda^{p-1-q}d\lambda \le \frac{q}{q - p}C^{\frac pq}, 
\end{align*}
Thus, the lemma holds with $K_{p,q} = q/(q - p) > 1$. 
\end{proof}

\section{Sub-Gaussian random variable}
\label{appendix:subgaussian}

Recall that a real random variable $X$, is called sub-Gaussian of order $\sigma^2$, if 
\begin{equation}
  \log E \big[e^{sX}\big] \le \frac{\sigma^2s^2}2, \quad \forall s \in \bR. 
\end{equation}
There is an elementary but useful condition for sub-Gaussian property. 

\begin{lem}[$\phi_2$ condition]
\label{lem:orlicz}
If $E [X] = 0$, and 
\begin{equation}
\label{orlicz}
  E \big[e^{cX^2}\big] \le C, 
\end{equation}
for some $c > 0$ and $C \ge 1$, then $X$ is sub-Gaussian of order $2Cc^{-1}$. 
\end{lem}

\begin{proof}
Since $E [X] = 0$, we have for any $s \in \bR$ that 
$$
  E \big[e^{sX}\big] = 1 + \sum_{k=2}^\infty \frac{E [(sX)^k]}{k!} \le 1 + \frac{s^2}2\sum_{k=0}^\infty \frac{|s|^kE [|X|^{k+2}]}{k!}. 
$$
The summation in the right-hand side is bounded by 
$$
  \frac{s^2}2E \big[X^2e^{|sX|}\big] \le \frac{s^2}2E \left[X^2\exp\left\{\frac {cX^2}2 + \frac{s^2}{2c}\right\}\right] 
$$
for any $c > 0$. 
With the elementary inequality $ye^y \le e^{2y}$, 
$$
  \frac{s^2}2E \left[X^2\exp\left\{\frac {cX^2}2 + \frac{s^2}{2c}\right\}\right] \le \frac{s^2}c\exp\left\{\frac{s^2}{2c}\right\}E \big[e^{cX^2}\big]. 
$$
Hence, by the condition \eqref{orlicz}, 
$$
  E \big[e^{sX}\big] \le 1 + \frac{Cs^2}c\exp\left\{\frac{s^2}{2c}\right\} \le \exp\left\{\frac{Cs^2}c\right\}. 
$$
As $s$ is arbitrary, the proof is completed. 
\end{proof}

Recall that in Lemma \ref{lem:main lemma} we need to bound the exponential integral of the absolute value of a sub-Gaussian variable. 
The general estimate is as follows. 

\begin{lem}
\label{lem:subgaussian absolute}
If $X$ is sub-Gaussian of order $\sigma^2$, then 
$$
  E \big[e^{s|X|}\big] \le \frac{1 + |s|}{1 - |s|}\exp\left\{\frac{\sigma^2|s|}2\right\}, \quad \forall s \in (-1, 1). 
$$
\end{lem}

\begin{proof}
By Chernoff's method, for any $\lambda > 0$, 
$$
  P(X \ge \lambda) \le E \left[\exp\left\{\frac{\lambda(X - \lambda)}{\sigma^2}\right\}\right] \le \exp\left\{\frac{\lambda^2}{2\sigma^2} - \frac{\lambda^2}{\sigma^2}\right\} = \exp\left\{- \frac{\lambda^2}{2\sigma^2}\right\}. 
$$
Since similar estimate holds for $P(X \le -\lambda)$, 
\begin{equation}
\label{subgaussian tail}
  P(|X| \ge \lambda) \le 2\exp\left\{-\frac{\lambda^2}{2\sigma^2}\right\}. 
\end{equation}
For $0 \le t < 1/(2\sigma^2)$, the integration-by-parts formula yields that 
$$
  E \big[e^{tX^2}\big] \le 1 + \int_0^\infty \frac d{d\lambda} \big(e^{t\lambda^2}\big)P(|X| \ge \lambda)d\lambda \le \frac{1 + 2t\sigma^2}{1 - 2t\sigma^2}. 
$$
Hence, for any $s \in [0, 1)$, 
$$
  E \big[e^{s|X|}\big] \le E \left[\exp\left\{\frac{sX^2}{2\sigma^2} + \frac{\sigma^2s}2\right\}\right] \le \frac{1 + s}{1 - s}\exp\left\{\frac{\sigma^2s}2\right\}. 
$$
The case $s \in (-1, 0)$ holds similarly. 
\end{proof}

\section{Local central limit theorem}
\label{appendix:local clt}

In this appendix, we state a local central limit theorem with expansions for the sum of independent, non-identically distributed random variables. 
It is used in the proof of equivalence of ensembles in Section \ref{sec:ee}. 

We work under the following setting. 
Suppose that $\pi$ is some Borel measure on $\bR$, and $f: \bR \to \bR$ is an integrable function. 
Assume for all $\tau \in \bR$ that 
$$
  G(\tau) \triangleq \log\int_\bR e^{\tau f(r)}\pi(dr) < \infty. 
$$
Denote by $\pi_\tau$ the tilted probability measure on $\bR$, given by 
$$
  \pi_\tau(dr) = \exp\{\tau f-G(\tau)\}\pi(dr). 
$$
Let $\Phi_\tau(\xi) = \int \exp\{i\xi(f - \int fd\pi_\tau)\}\pi_\tau(dr)$ be the characteristic function of $f$. 
For all $K > 0$, we assume the following conditions with a constant $M_K$: 
\begin{itemize}
\item[(\romannum1)] $G$ is four times differentiable on $\bR$, and 
$$
  G''(\tau) > M_K^{-1},\ \big|G^{(\ell)}(\tau)\big| < M_K,\ \forall \tau \in [-K, K],\ \ell = 0, 1, 2, 3, 4. 
$$
\item[(\romannum2)] $|\Phi_\tau(\xi)| < M_K(1 + |\xi|)^{-1}$ for all $\xi \in \bR$ and $\tau \in [-K, K]$; 
\item[(\romannum3)] $\Phi_\tau$ is four times differentiable on $\bR$ for all $\tau \in \bR$, and 
$$
  \forall \epsilon > 0,\ \exists \delta = \delta(\epsilon, K) > 0,\ \text{s.t.}\ \big|\Phi_\tau^{(\ell)}(\xi) - \Phi_\tau^{(\ell)}(0)\big| < \epsilon, 
$$
for all $|\xi| < \delta$, $\tau \in [-K, K]$ and $\ell = 0$, $1$, $2$, $3$, $4$. 
\end{itemize}

Given $\vt = (\tau_1, \ldots, \tau_n)$, define the inhomogeneous product measure 
$$
  \mu_n(d\bfr) = \mu_n(\vt; d\bfr) = \prod_{j=1}^n \pi_{\tau_j}(dr_i), \quad \bfr = (r_1, \ldots, r_n) \in \bR^n. 
$$
Define $u_\ell = u_\ell(\vt) > 0$ for $\ell = 1$, $2$, $3$, $4$ via the formula 
$$
  |u_\ell|^\ell = \frac1n\sum_{j=1}^n G^{(\ell)}(\tau_j). 
$$
Observe that $u_1 = E_{\mu_n}[\bar r]$ and $u_2^2 = E_{\mu_n} [(\bar r - u_1)^2]$, where $\bar r = n^{-1}\sum r_j$. 

The local central limit theorem is stated as follows. 
Let $\phi$ be the standard Gaussian density, and $\{H_j; j \ge 0\}$ be the group of Hermite polynomials: 
$$
  \phi(x) = \frac1{\sqrt{2\pi}}e^{-\frac{x^2}2}, \quad H_j(x) = (-1)^j e^{\frac{x^2}2}\frac{d^j}{dx^j}\left[e^{-\frac{x^2}2}\right]. 
$$
In particular, $H_3 = x^3 - 3x$, $H_4 = x^4 - 6x^2 +3$ and $H_6 = x^6 - 15x^4 +45 x^2 - 15$. 

\begin{lem}
\label{lem:local clt}
Assume that $\tau_j \in [-K, K]$ for all $1 \le j \le n$. 
Let $g_n(\vt; \cdot)$ be the density function with respect of $\mu_n$ of the random variable 
$$
  \frac1{u_2(\vt)\sqrt n}\sum_{j=1}^n \big(r_j - u_1(\vt)\big). 
$$
For any $\epsilon > 0$, there exists $N = N(\epsilon, K, M_K)$ sufficiently large, such that if $n \ge N$, then the following estimate holds uniformly for $x \in \bR$: 
$$
  \left|g_n(\vt; x) - \phi(x)\left[1 + \frac1{\sqrt n}Q_{n,1}(x) + \frac1nQ_{n,2}(x)\right]\right| < \frac Cn\left(\epsilon + \frac1{\sqrt n}\right) 
$$
where $C = C(M_K)$ is a constant and $Q_{n,1}$, $Q_{n,2}$ are given by 
$$
  Q_{n,1} = \frac1{3!}\left(\frac{u_3}{u_2}\right)^3H_3, \quad Q_{n,2} = \frac1{4!}\left(\frac{u_4}{u_2}\right)^4H_4 + \frac1{2(3!)^2}\left(\frac{u_3}{u_2}\right)^6H_6. 
$$
\end{lem}

Lemma \ref{lem:local clt} can be proved following \cite[Theorem \Romannum{16}.2.2, pp. 535]{Feller71}. 
Here we briefly sketch the proof to emphasize the dependence of $(N, C)$ on $\epsilon$, $K$ and $M_K$. 

\begin{proof}
By the definition of characteristic function $\Phi_\tau$, 
$$
  g_n(\vt; x) = \frac1{2\pi}\int_\bR e^{-ix\xi}\prod_{j=1}^n \Phi_{\tau_j}\left(\frac\xi{u_2\sqrt n}\right)d\xi. 
$$
Let us define $\Delta_n = \Delta_n(\vt; \xi)$ for each $\xi \in \bR$ by 
$$
  \Delta_n(\vt; \xi) = \prod_{j=1}^n \Phi_{\tau_j}\left(\frac\xi{u_2\sqrt n}\right) - \exp\left\{-\frac{\xi^2}2\right\}\left[1 + P_n(i\xi) + \frac12P_n^2(i\xi)\right], 
$$
where $P_n = P_n(\vt; \cdot)$ is the polynomial given by 
$$
  P_n = \frac1{3!\sqrt n}\left(\frac{u_3}{u_2}\right)^3x^3 + \frac1{4!n}\left(\frac{u_4}{u_2}\right)^4x^4. 
$$
From the definition of Hermite polynomials, it suffices to prove that 
$$
  \int_\bR \big|\Delta_n(\vt; \xi)\big|d\xi \le \frac Cn\left(\epsilon + \frac1{\sqrt n}\right). 
$$
For any $\epsilon > 0$, Taylor's theorem yields that there is $\delta = \delta(\epsilon, K) > 0$, such that 
$$
  \left|\log\Phi_\tau(\xi) + \frac{G''(\tau)}2\xi^2 - \sum_{\ell=3}^4 \frac1{\ell!}G^{(\ell)}(\tau)(i\xi)^\ell\right| < \epsilon\xi^4, 
$$
for all $|\xi| < \delta$ and $\tau \in [-K, K]$. 
Therefore, when $|\xi| < \delta u_2\sqrt n$, 
$$
  \left|\sum_{j=1}^n \log\Phi_{\tau_j}\left(\frac\xi{u_2\sqrt n}\right) + \frac{\xi^2}2 - P_n(i\xi)\right| < \frac{\epsilon\xi^4}{u_2^4n}. 
$$
Without loss of generality we can choose $\delta < 1$, so that 
$$
  |P_n(i\xi)| %\le \left(\frac1{3!}\frac{u_3^3}{u_2^3} + \frac{|\xi|}{4!\sqrt n}\frac{u_4^4}{u_2^4}\right)\frac{|\xi|^3}{\sqrt n} 
  < \left(\frac{u_3^3}{3!u_2^3} + \frac{\delta u_4^4}{4!u_2^3}\right)\frac{|\xi|^3}{\sqrt n} < \frac{C_1|\xi|^3}{\sqrt n}, 
$$
with some $C_1 = C_1(M_K)$. 
Using the elementary inequality 
$$
  \left|e^x - 1 - x' - \frac{(x')^2}2\right| \le e^{\max\{|x|,|x'|\}}\big(|x - x'| + |x'|^3\big), \quad \forall x, x' \in \bR, 
$$
we obtain that when $|\xi| < \delta u_2\sqrt n$, 
$$
  \big|\Delta_n(\vt, \xi)\big| < \exp\left\{\frac{\epsilon\xi^4}{u_2^4n} + \frac{C_1|\xi|^3}{\sqrt n} - \frac{\xi^2}2\right\}\frac1n\left(\frac{\epsilon\xi^4}{u_2^4} + \frac{C_1^3|\xi|^9}{\sqrt n}\right). 
$$
By furthermore choosing $\delta = \delta(\epsilon, K, M_K)$ sufficiently small, we get 
$$
  \big|\Delta_n(\vt, \xi)\big| %< \frac1n\exp\left\{\left(\frac{\epsilon\delta^2}{u_2^2} + C_1\delta u_2 - \frac12\right)\xi^2\right\}\left(\frac{\epsilon\xi^4}{u_2^4} + \frac{C_1^3|\xi|^9}{u_2^9\sqrt n}\right) 
  < \frac{C_2}n\exp\left\{-\frac{\xi^2}4\right\}\left(\epsilon\xi^4 + \frac{|\xi|^9}{\sqrt n}\right), 
$$
with some $C_2 = C_2(M_K)$ on the set $\{|\xi| < \delta u_2\sqrt n\}$. 
From the estimate above, we have some constant $C = C(M_K)$, such that for all $n \ge 1$, 
$$
  \int_{|\xi|<\delta u_2\sqrt n} \big|\Delta_n(\vt, \xi)\big|d\xi \le \frac Cn\left(\epsilon + \frac1{\sqrt n}\right). 
$$
On the remaining set $\{|\xi| \ge \delta u_2\sqrt n\}$, by (\romannum2) we have that 
$$
  \big|\Delta_n(\vt, \xi)\big| < \frac{M_K^n}{(1 + |\xi|)^n} + \exp\left\{-\frac{\xi^2}2\right\}\left[1 + P_n + \frac12P_n^2\right]. 
$$
Hence, we can choose $N = N(\delta, M_K)$, such that for all $n \ge N$, 
$$
  \int_{|\xi|\ge\delta u_2\sqrt n} \big|\Delta_n(\vt, \xi)\big|d\xi < \frac1{n^{3/2}}. 
$$
The proof is then completed. 
\end{proof}

\end{appendices}

%%%%%%%%%%%%%%%%%%%%%%%%%%%%%%%%%%%%%%%%%
%references with default bookmark

\titleformat{\section}[hang]	
{\bfseries\large}{}{0em}{}[]
\titlespacing*{\section}{0em}{2em}{1.5em}

\section{References}	
\renewcommand{\section}[2]{}
\bibliography{[BIB]common.bib,[BIB]chain_of_oscillators.bib}

\vspace{1em}
\noindent{\large Lu \textsc{Xu}}

\vspace{0.5em}
\noindent Gran Sasso Science Institute\\
Viale Francesco Crispi n.7, L'Aquila, Italy\\
{\tt lu.xu@gssi.it}

\end{document}